\documentclass[11pt]{article}
\usepackage[margin=1in]{geometry}

\usepackage{amsmath,amssymb,amsthm,mathrsfs}
\usepackage{bm}
\usepackage{amscd}
\usepackage{enumitem}
\usepackage{graphicx}
\usepackage[dvipsnames]{xcolor}
\usepackage[colorlinks=true,linkcolor=blue,citecolor=blue,urlcolor=blue]{hyperref}

\theoremstyle{plain}

\numberwithin{equation}{section}
\newtheorem{theorem}{Theorem}[section]
\newtheorem{proposition}[theorem]{Proposition}
\newtheorem{lemma}{Lemma}[section]
\newtheorem{corollary}[theorem]{Corollary}

\theoremstyle{definition}

\newtheorem{definition}[theorem]{Definition}
\newtheorem{remark}{Remark}[section]

\allowdisplaybreaks[4] 

\def\R{\mathbb{R}}

\def\dist{\operatorname{dist}}

\def\div{\operatorname{div}}

\def\<{\langle} \def\>{\rangle}

\def\L{\mathcal{L}}

\newcommand{\Bu}{{\boldsymbol{u}}}

\newcommand{\Be}{{\boldsymbol{e}}}

\newcommand{\Bx}{\boldsymbol{x}}
\def\dd{\mathop{}\!\mathrm{d}}

\newcommand{\ignore}[1]{}

\newcommand{\stkout}[1]{\ifmmode\text{\sout{\ensuremath{#1}}}\else\sout{#1}\fi}
\usepackage{algorithm, algorithmic}
\title{Asymptotic long-time behavior of Darcy--Boussinesq convection in layered porous media with narrow transition zones}
\author{%
Kaijian Sha$^{1}$\thanks{Email: kjsha11@eitech.edu.cn}
\and
Xiaoming Wang$^{1}$\thanks{Email: wxm.math@outlook.com. Corresponding author.}
\and
Hao Wu$^{2}$\thanks{Email: haowufd@fudan.edu.cn}
\\[0.5em]
\small $^1$ $^1$ School of Mathematical Sciences, Eastern Institute of Technology, Ningbo, China\\
\small $^2$ $^2$ School of Mathematical Sciences, Fudan University, Shanghai 200433, China
}
\date{}

\begin{document}

\maketitle

\begin{abstract}
We study the asymptotic long-time behavior of Darcy--Boussinesq convection in layered porous media with narrow transition zones in the material properties. As the transition-layer width tends to zero, we prove the upper semi-continuous convergence of the global attractor, invariant measure, and Nusselt number to their counterparts in the limiting sharp-interface model. We also show that the global attractors have finite fractal dimensions, with an explicit upper bound uniform in the transition-layer width.
The analysis combines a carefully designed background temperature/contaminant profile together with a novel choice of phase space that ensures global well-posedness of the model and asymptotic compactness of the solution semigroup, and a new interpolation inequality. The phase space is associated with fractional powers of the principal elliptic operator with discontinuous coefficients. These results provide a rigorous long-time validation of the sharp-interface Darcy--Boussinesq model and extend our earlier finite-time convergence theory (H. Dong and X. Wang, SIAM J. Appl. Math. 85 (2025), 1621--1642) to the long-time regime.
\end{abstract}

\medskip


\section{Introduction}

Fluid convection in porous structures plays a fundamental role in both natural phenomena and technological applications. This subject has attracted sustained attention over the decades; see Nield and Bejan \cite{nield2017convection} for an overview. Many mathematical studies have analyzed convection in porous settings, including well-posedness, regularity, long-time behavior, and transport bounds \cite{doering1998jfm, fabrie1986aam, fabrie1996aa, ly1999jns, oliver2000gevrey, otero2004jfm, WCC2015, DPZS2022, CNW2025}. With the exception of recent contributions such as \cite{CNW2025, DW2025}, this body of work has focused primarily on media with uniform (spatially constant) material properties.

In practice, porous formations are often stratified or otherwise heterogeneous. Variations in permeability and transport coefficients occur frequently in distinct layers, arising from geological processes (e.g., sediment deposition, tectonic activity, dissolution) or from engineered design. A prominent application is the geological storage of carbon dioxide (CO$_2$), where layering can markedly influence stability, plume morphology, and convective transport; see, for instance, \cite{bickle2007modelling, huppert2014ar, hewitt2014jfm, hewitt2020jfm, hewitt2022jfm, mckibbin1980jfm, mckibbin1981heat, mckibbin1983thermal, sahu2017tansp, salibindla2018jfm, wooding1997convection, DePaoli2021, HFCJ2012}. The review by Huppert and Neufeld \cite{huppert2014ar} summarizes much of the literature, while \cite{BS2024} documents more recent progress on the role of layering in modifying---and, in some regimes, enhancing---convective transport.

A widely used simplification is the sharp-interface description, in which material parameters are modeled as piecewise constant—uniform within each layer and discontinuous across interfaces—and adjacent layers are coupled through transmission conditions \cite{hewitt2014jfm, hewitt2020jfm, hewitt2022jfm, mckibbin1980jfm, mckibbin1981heat, mckibbin1983thermal, CNW2025}. This formulation is mathematically convenient and often effective at macroscopic scales. However, it is also idealized: at microscopic scales, diffusive and mixing mechanisms smooth abrupt parameter jumps, so infinitely sharp transitions are not physically realizable. Moreover, while equations for homogeneous porous flow admit rigorous derivations (e.g., via homogenization), a systematic derivation of the sharp-interface model and its interfacial conditions is subtler in layered settings. In particular, the precise form of the transmission conditions---especially which components of velocity and stress should be continuous across an interface---can depend on the upscaling assumptions and the modeling regime and, therefore, merits mathematical justification.

A more faithful representation allows for diffuse interfaces, where coefficients vary smoothly across thin transition layers. Following \cite{diffuse98, SMCTSH17}, we refer to this as the diffuse-interface model. A natural theoretical question is whether the sharp-interface formulation (together with its transmission conditions) can be justified as the singular limit of the diffuse-interface model as the transition thickness tends to zero.

In previous work \cite{DW2025}, we established convergence of the diffuse-interface model to the sharp-interface model on finite time intervals, under an initial regularity assumption formulated in terms of the system’s linear operator. Despite near-singular gradients in thin transition zones, the sharp-interface dynamics is recovered in this short-time regime. For many applications, however, the quantities of interest are intrinsically long-time or statistical, for instance, invariant measures describing statistically steady states and long-time averaged transport metrics such as the Nusselt number.

The purpose of this study is to precisely address the asymptotic behavior of the model in the long-time regime. We prove that, as the thickness of the transition layer $\varepsilon\to 0$, the global attractors, invariant measures, and the Nusselt number of the diffuse-interface model \eqref{Diffuse}–\eqref{bc0-epsilon} converge upper semicontinuously to their counterparts for the sharp-interface model \eqref{Sharp}–\eqref{bc0}. In addition, we show that the global attractors have finite fractal dimensions, and we obtain an explicit upper bound that is uniform in $\varepsilon$.

The analysis differs substantially from \cite{DW2025}. Long-time theory requires a phase space that is (i) regular enough to control the nonlinear term globally in time and (ii) sufficiently large to establish asymptotic compactness of the solution semigroup uniformly with respect to $\varepsilon$. To this end, we introduce a phase space defined via fractional powers of the principal elliptic operator of the sharp-interface model (with discontinuous coefficients). A key new ingredient is an interpolation inequality linking this fractional-power space to standard Sobolev spaces; it yields global-in-time well-posedness and uniform-in-$\varepsilon$ \textit{a priori} bounds. Asymptotic compactness is obtained from uniform-in-time $H^1$-estimates via a uniform Gronwall-type argument. Finally, the finite-dimensionality of the attractor follows from a squeezing-type property of the semigroup, again with constants uniform in the transition-layer thickness. Altogether, the results provide a rigorous long-time justification of the sharp-interface Darcy--Boussinesq model and extend the finite-time convergence theory of \cite{DW2025} to the long-time regime.

The remainder of the paper is organized as follows. Section 2 introduces the sharp- and diffuse-interface formulations. Section 3 establishes preliminary estimates. Section 4 proves existence of global attractors with bounds uniform in the interface thickness. Section 5 proves upper semicontinuity for global attractors, invariant measures, and the Nusselt number. Section 6 derives dimension estimates. Section 7 concludes with remarks, and the Appendix collects auxiliary estimates and the proof of the new interpolation inequality.

\section{Model Description}

We consider convection in a three-dimensional layered domain $\Omega = (0,L)^2\times(-H,0)$, $L, H>0$. The domain $\Omega$ is divided into $l$ flat `layers' or `strips' by $l-1$ interfaces located at $z= z_j \in (-H,0), j=1,2,\cdots,l-1$. In addition, we denote $z_0=0, z_l=-H$. The $j$-th layer is given by  
\[\Omega_j: = \{\Bx\in \Omega:~\Bx \cdot \Be_z\in (z_{j},z_{j-1})\},\quad \text{ for }j=1,2,\cdots,l,\]
where $\bm{x}=(x,y,z)$ and $\Be_{z}$ stands for the unit vector in the $z$-direction. 

\subsection{The sharp interface model}
Let $T>0$. The Darcy--Boussinesq system describing   convection in the layered domain $\Omega$ takes the following form (see \cite{nield2017convection}):
\begin{equation}\label{Sharp}
\left\{\begin{aligned}
  &\Bu=-\frac{K}{\mu}\left(\nabla P + \rho_{0}(1+\alpha\phi)g\Be_{z}\right),\\
  &\nabla \cdot \Bu =0,\\
  &b\partial_t \phi  + \Bu\cdot\nabla\phi - \nabla \cdot(bD\nabla\phi)=0,
\end{aligned}\right.\qquad \text{in}\ \Omega\times (0,T).
\end{equation}
Here, the unknowns $\Bu:\Omega\times (0,T)\to \mathbb{R}^3$, $P:\Omega\times (0,T)\to \mathbb{R}$, and $\phi: \Omega \times (0,T)\to \mathbb{R}$ are the fluid velocity, pressure, and concentration, respectively. The constants $\rho_{0}$, $\alpha$, $\mu$, $g$ are the reference fluid density, expansion coefficient, dynamic viscosity, and the gravity acceleration. In addition, $K$, $b$, $D$ represent the permeability, porosity, and diffusivity coefficients, respectively. 

For the idealized sharp interface model, 
in each layer $\Omega_{j}$, the permeability, porosity, and diffusivity coefficients are assumed to be constant, that is,
\[
K=K(\boldsymbol{x})=K_{j}, \quad b=b(\boldsymbol{x})=b_{j}, \quad D=D(\boldsymbol{x})=D_{j}, \quad \forall\, \boldsymbol{x}\in\Omega_{j}, \quad 1\leq j\leq l,
\]
for a set of positive constants $\{K_{j},b_{j},D_{j}\}_{j=1}^{l}$. On the interfaces, we assume
\begin{equation}
\label{interface-1}
  \Bu\cdot\Be_{z},\ P,\ \phi,\text{ are continuous at } z=z_{j},\quad 1\leq  j\leq  l-1,
\end{equation}
%
%
In addition, the system \eqref{Sharp} is supplemented with the initial condition
\begin{equation}\label{initialdata}
  \phi|_{t=0}=\phi_{0}
\end{equation}
and the boundary conditions
\begin{equation}\label{bc0}
\Bu\cdot \Be_z|_{z=0,-H}= 0,\quad \phi|_{z=0}=c_0, \quad \phi|_{z=-H}=c_{-H},
\end{equation}
together with periodicity in the horizontal directions $(x,y)$. 


\subsection{The diffuse interface model}
While the sharp interface model allows the physical parameters to change suddenly at the interfaces $z=z_j$, a physically more realistic setting would allow the material parameters to change gradually, albeit over a thin region of thickness $\varepsilon$, due to local diffusion.
We assume that
\[ 0<\varepsilon \ll  z_j-z_{j+1},\quad\forall\, j=0,1,\cdots, l-1.\]
This implies that the gradual transition layers are much thinner than the macroscopic layers. In the diffuse interface framework, the governing equations take the following form
\begin{equation}\label{Diffuse}
\left\{\begin{aligned}
  &\Bu^\varepsilon=-\frac{K^\varepsilon}{\mu}\left(\nabla P^\varepsilon + \rho_{0}(1+\alpha\phi^\varepsilon)g\Be_{z}\right),\\
  &\nabla \cdot \Bu^\varepsilon =0,\\
  &b^\varepsilon\partial_t \phi^\varepsilon  + \Bu^\varepsilon\cdot\nabla\phi^\varepsilon - \nabla \cdot(b^\varepsilon D^\varepsilon\nabla\phi^\varepsilon)=0,
\end{aligned}\right.\qquad \text{in}\ \Omega\times (0,T).
\end{equation}
Here, the unknowns $\Bu^\varepsilon:\Omega\times (0,T)\to \mathbb{R}^3$, $P^\varepsilon:\Omega\times (0,T)\to \mathbb{R}$, and $\phi^\varepsilon: \Omega \times (0,T)\to \mathbb{R}$ are the fluid velocity, pressure, and concentration, respectively. $K^\varepsilon$, $b^\varepsilon$, $D^\varepsilon$ represent the permeability, porosity, and diffusivity coefficients that are essentially constant in each of the $\Omega_j$, but with a narrow thin layer of width $2\varepsilon$ that allows a smooth transition between the neighboring strips. More precisely, $K^\varepsilon$, $b^\varepsilon$, $D^\varepsilon$ are smooth over $\Omega$ and 
\begin{align*}
& K^\varepsilon=K_j,~b^\varepsilon= b_j,~D^\varepsilon= D_j,\quad 
\forall\, z\in [z_{j}+\varepsilon,z_{j-1}-\varepsilon],\quad 2\leq  j\leq  l-1, \\
& K^\varepsilon=K_1,~b^\varepsilon= b_1,~D^\varepsilon= D_1,\quad \forall\, z\in [z_1+\varepsilon,z_0],\\
& K^\varepsilon=K_l,~b^\varepsilon= b_l,~D^\varepsilon= D_l,\quad \forall\, z\in [z_l,z_{l-1}-\varepsilon].
\end{align*} 
For simplicity, we assume that $K^\varepsilon$, $D^\varepsilon$, $b^\varepsilon$ are continuous piecewise linear functions along the $z$-direction, for example,
\[K^\varepsilon(z) = K_{j-1} +(z-z_j+\varepsilon) \frac{K_j-K_{j-1}}{2\varepsilon},\quad \forall\,z\in (z_j-\varepsilon,z_j+\varepsilon),\]
for $j=1,2,\cdots,l-1$. 
The system \eqref{Diffuse} is supplemented with the initial condition
\begin{equation}\label{initialdata-epsilon}
  \phi^\varepsilon|_{t=0}=\phi_{0}^\varepsilon
\end{equation}
and the boundary conditions
\begin{equation}\label{bc0-epsilon}
\Bu^\varepsilon\cdot \Be_z|_{z=0,-H}= 0,\quad 
\phi^\varepsilon|_{z=0}=c_0, \quad \phi^\varepsilon|_{z=-H}=c_{-H}, 
\end{equation}
together with periodicity in the horizontal directions $(x,y)$.

\subsection{Homogenization of boundary conditions and the reduced systems}
For the sake of simplicity, in the subsequent analysis, we assume 
$$\mu = 1, \quad \rho_{0} \alpha g =1, \quad b=1.$$ 
Next, to deal with the nonhomogeneous boundary condition of $\phi$ and $\phi^\varepsilon$, we introduce a smooth function $\phi_b(z;\delta)$ on $[-H,0]$ satisfying
\begin{equation}\label{phi_b1}
\phi_b(z;\delta)= \left\{\begin{aligned}
& c_0, && z = 0,\\
& \frac{c_0+c_{-H}}{2}&&z \in (-H+\delta,-\delta),\\
& c_{-H}, && z=-H
\end{aligned}\right.
\end{equation}
and 
\begin{equation}\label{phi_b2}
|\phi_b'| \leq \frac{c_\Delta}{\delta}, 
\quad 
|\phi_b''| \leq \frac{2c_\Delta}{\delta^2},
\end{equation}
where $c_\Delta := |c_0-c_{-H}|$ and $0<\delta\ll 1$ is a small parameter to be determined later.

Taking
\[ \psi =\phi -  \phi_b  \quad 
\text{ and }  \quad p = P -\rho_0 g z - \alpha \rho_0 \int_{-H}^z\phi_b(\tau)\dd \tau,
\]
we can write the sharp interface model \eqref{Sharp} as 
\begin{equation}\label{Sharps}
      \left\{\begin{aligned}
        &\Bu=- K \left(\nabla p  + \psi \Be_{z}\right),\\
        &\nabla \cdot \Bu =0,\\
        & \partial_t \psi + \Bu\cdot\nabla \psi  +  \phi_b'\Bu \cdot \Be_z   - \nabla \cdot (D\nabla\psi)=D \varphi_b'',
      \end{aligned}\right.  
  \end{equation}
which is supplemented with the interfacial boundary conditions 
\begin{equation}\label{interface1}
  \Bu\cdot\Be_{z},\ p,\ \psi\ \ \text{ are continuous at } z=z_{j},\ 1\leq  j\leq  l-1,
\end{equation}
the initial condition
\begin{equation}\label{initial}
  \psi|_{t=0}=\psi_0:=\phi_{0} -\phi_b,
\end{equation}
the homogeneous boundary conditions
\begin{equation}\label{bc}
    \Bu\cdot \Be_z|_{z=0,-H}= 0,\quad \psi|_{z=0}=0, \quad \psi|_{z=-H}=0,    
\end{equation}
and periodicity in the horizontal directions $(x,y)$.
\begin{remark}
 We note that $\phi_b'$ is a smooth function on $[-H,0]$ that vanishes on $[-H+\delta, -\delta]$. Then $D \phi_b' $ is also smooth provided that $\delta$ is sufficiently small. Consequently, it holds 
      $\partial_z(D \phi_b') = D \phi_b''$. 
\end{remark}

Using the same background function $\phi_b(z;\delta)$, the diffuse interface model \eqref{Diffuse} can be homogenized as  
\begin{equation}\label{Diffuse-s}
      \left\{\begin{aligned}
        &\Bu^\varepsilon=- K^\varepsilon\left(\nabla p^\varepsilon  +  \psi^\varepsilon \Be_{z}\right),\\
        &\nabla \cdot \Bu^\varepsilon =0,\\
        & \partial_t \psi^\varepsilon + \Bu^\varepsilon\cdot\nabla \psi^\varepsilon  +  \phi_b'\Bu^\varepsilon \cdot \Be_z   - \nabla \cdot (D^\varepsilon\nabla\psi^\varepsilon)= D^\varepsilon \varphi_b'',
      \end{aligned}\right.  
  \end{equation} 
which is supplemented with the initial condition 
\begin{equation}\label{initial-epsilon}
  \psi^\varepsilon|_{t=0}=\psi_0^\varepsilon :=\phi_{0}^\varepsilon-\phi_b,
\end{equation}
the homogeneous boundary conditions 
\begin{equation}\label{bc-epsilon}
    \Bu^\varepsilon\cdot \Be_z|_{z=0,-H}= 0,\quad 
    \psi^\varepsilon|_{z=0}=0, \quad \psi^\varepsilon|_{z=-H}=0,
\end{equation}
and periodicity in the horizontal directions $(x,y)$.

\section{Well-posedness}

\subsection{Preliminaries}
First, we recall some functional settings introduced in \cite{CNW2025}. Let $X$ be a real Banach space. We denote its norm by $\|\cdot\|_X$, its dual space by $X^*$, and the duality pairing by $\langle \cdot,\cdot\rangle_{X^*,X}$. For $1\leq r \leq \infty$ and $k\in \mathbb{N}$, let $L^r (\Omega)$ and $W^{k,r}(\Omega)$ denote the usual Lebesgue spaces and Sobolev spaces on the domain $\Omega$, both with periodicity in the horizontal directions $(x,y)$. The inner product in $L^2(\Omega)$ is denoted by $(\cdot,\cdot)$ for simplicity. When $r=2$, the notation $H^k(\Omega)=W^{k,2}(\Omega)$ will be used. We denote by $W^{s,p}(\Omega)$ for $s\geq 0$ and $1\leq r \leq \infty$ the fractional Sobolev--Slobodeckij spaces. When $p=2$, we denote  $H^s(\Omega)=W^{s,2}(\Omega)$.  
Define
\begin{align}
\mathbb{V}& := \{\psi\in C(\overline{\Omega}),~\psi|_{\Omega_j} \in C^\infty(\overline{\Omega_j})\text{ for any } j=1,2,\cdots,l, \notag\\
&\qquad \psi \ \text{is periodic in the horizontal directions and satisfies }\eqref{interface1},\,\eqref{bc}\},
\notag 
\end{align}
and $\mathcal{H}$, $\mathcal{V}$ be the closure of $\mathbb{V}$ in the $L^2$ and $H^1$ norms, respectively. Then we recognize $\mathcal{H}=L^2(\Omega)$ and $\mathcal{V}=H_{(0)}^1(\Omega)$, where $H_{(0)}^1(\Omega)$ denotes the subspace of $H^1(\Omega)$ for functions that vanish at
$z = 0, -H$ and are periodic in the horizontal directions $(x,y)$. Using the boundedness of $D$ and the Poincar\'{e} inequality (see Lemma \ref{lemmaA5}), we can define the equivalent norm in $\mathcal{V}$ as $\|\psi\|_{\mathcal{V}}=\|\sqrt{D}\nabla \psi\|_{L^2(\Omega)}$. The corresponding inner product is given by $(\cdot,\cdot)_1=(D\nabla \cdot,\nabla \cdot)$. For simplicity, we denote the dual product between $\mathcal{V}^*$ and $\mathcal{V}$ by $\langle\cdot,\cdot \rangle$. Next, we define the higher-order space 
\begin{equation}\nonumber
\mathcal{W} = \{\psi \in \mathcal{V}:~\partial_x \psi,\,\partial_y \psi \in H^1(\Omega),~D\partial_z \psi \in H^1(\Omega)\}
\end{equation}
endowed with the norm
\begin{equation}\nonumber
  \|\psi\|_{\mathcal{W}}^2 = \|\psi\|_{H^1(\Omega)}^2 
  +\|\partial_x \psi\|_{H^1(\Omega)}^2 
  +\|\partial_y \psi\|_{H^1(\Omega)}^2 
  +\|D\partial_z \psi\|_{H^1(\Omega)}^2.
\end{equation}
We note that $\mathcal{W}$ is different from $H^2(\Omega)$ in general. If $D$ contains a jump discontinuity at $z=z_j$, $j=1,\cdots,l-1$, we can verify that $\mathcal{W}\cap H^2(\Omega)=\{\psi\,:\, \psi\in H^2(\Omega),\ \partial_z\psi|_{z=z_j}=0,\ j=1,\cdots,l-1\}$, see \cite[Remark 2.3]{CNW2025}. 
Concerning the fluid velocity, we introduce the space 
\begin{align}
\widetilde{\mathbb{V}} & := \{ \Bu \in C(\overline{\Omega})^3,~\Bu|_{\Omega_j} \in C^\infty(\overline{\Omega_j})^3\text{ for any } j=1,2,\cdots,l,\nonumber\\
&\qquad \Bu \text{ is periodic in the horizontal directions and satisfies }\eqref{Sharps}_2, \eqref{interface1}, \eqref{bc}\},
\notag 
\end{align}
and set $\bm{H}$ to be the closure of $\widetilde{\mathbb{V}}$ in the $L^2$-norm. 

Next, we introduce the notion of weak solutions of the sharp interface model \eqref{Sharps}--\eqref{bc}.

\begin{definition}\label{def-Sharp-weak}
Let $T>0$. For any given initial datum $\psi_0\in  \mathcal{H}$, a triple $(\Bu, p, \psi)$ is called a weak solution to the problem \eqref{Sharps}--\eqref{bc} on $[0,T]$, if the following properties are satisfied: 
\begin{enumerate}
  \item $\psi \in L^\infty(0,T;\mathcal{H}) \cap L^2(0,T;\mathcal{V})\cap H^1(0,T;\mathcal{V}^*)$, 
  \begin{equation}\nonumber
    (\psi(t_2), \varphi) - ( \psi(t_1), \varphi)  = \int_{t_1}^{t_2}(\Bu\cdot\nabla \psi, \varphi)  +  (\phi_b'\Bu \cdot \Be_z,\varphi)  + (D\nabla\psi,\nabla\varphi)-(D \phi_b'',\varphi)\dd t,
  \end{equation}
  for any $t_1,t_2\in [0,T]$ and $\varphi\in \mathcal{V}$, moreover, $\psi(0)= \psi_0$ in $\mathcal{H}$;
\item $p \in L^2(0,T;\mathcal{V})$ is a weak solution to the elliptic problem
  \begin{equation}\label{pressure}
  \left\{\begin{aligned}
&-\nabla \cdot (K\nabla p) = \nabla \cdot (K\psi \Be_{z}), \quad \text{ in } \Omega,\\
&\partial_z p(x,y,-H)  = \partial_z p(x,y,0)  = 0,
  \end{aligned}\right. 
\end{equation}
and $p$ is periodic in the horizontal directions $(x,y)$; 
\item $\Bu \in L^2(0,T;\boldsymbol{H})$ satisfies \eqref{Sharps}$_1$ almost everywhere in $\Omega \times (0,T)$.
\end{enumerate}
\end{definition}

Applying the Faedo--Galerkin method and the energy method, one can obtain the following result on the existence, uniqueness, and regularity of solutions to the sharp interface model \eqref{Sharps}--\eqref{bc}, see \cite[Theorem 3.1, Theorem 4.2]{CNW2025}.
\begin{proposition}\label{prop1}
  For every $\psi_0\in \mathcal{H}$, problem \eqref{Sharps}--\eqref{bc} admits a global weak solution $(\Bu,p,\psi)$ on $[0,T]$ in the sense of Definition \ref{def-Sharp-weak}. In addition, if $\psi_0\in \mathcal{V}$, then the solution is unique and satisfies 
\begin{equation}\nonumber
\psi \in L^\infty(0,T;\mathcal{V}) \cap L^2(0,T;\mathcal{W}) \cap H^1(0,T;\mathcal{H})\quad \text{for any}\ T>0.
\end{equation} 
\end{proposition}
The regularity of $\bm{u}$ can easily be deduced from the regularity of $p$ and $\psi$. For any $\psi\in \mathcal{H}$, the existence of a weak solution $p$ to the elliptic problem \eqref{pressure} is guaranteed by the Lax--Milgram theorem. Uniqueness can be achieved if we restrict the solution $p$ to the subspace of $\mathcal{V}$ with zero mean. For a piecewise smooth function $\psi$ (smooth in each layer $\Omega_j$), one can show that $p$ is also piecewise smooth. For example, if $\psi\in \mathcal{V}$, then $p$ is piecewise in $H^2(\Omega)$. Moreover, we have the following estimate for $(\bm{u},p)$ (see \cite[Lemma 3.1]{CNW2025}):
\begin{lemma}\label{estimate-p}
If $\psi \in L^r(\Omega)$ for $r\in (1,\infty)$, the solution $p$ of the elliptic problem \eqref{pressure} satisfies $p\in W^{1,r}(\Omega)$ and
\[
\|p\|_{W^{1,r}(\Omega)} \leq C_p \|\psi\|_{L^r(\Omega)},
\]
where the positive constant $C_p=C_p(K_j,r)$. In view of \eqref{Sharps}$_1$, we also have 
\[\|\bm{u}\|_{L^r(\Omega)}\leq C \|\psi\|_{L^r(\Omega)}.\]
\end{lemma}
\begin{remark}
When $r=2$, the proof of Lemma \ref{estimate-p} is standard. For the case $r\neq 2$, we refer to \cite[Theorem 2.6]{DL2021}.
\end{remark}

\subsection{Well-posedness in $\mathcal{H}^s$}
Due to the lack of uniqueness in the three-dimensional case, one cannot define a (strongly) continuous semigroup corresponding to the weak solution $\psi(t)$ in the phase space $\mathcal{H}$. Therefore, we shall work within a more regular Hilbert space framework. 

Define $$\mathcal{L}\psi = -\nabla\cdot (D\nabla \psi),\quad \forall\, \psi \in \mathcal{W}.$$ 
It follows from \cite[Lemma 2.1]{CNW2025} that the operator $\mathcal{L}$ is a self-adjoint positive operator, and possesses a sequence of eigenvalues $\{\lambda_k\}_{k=1}^\infty$ where $\lambda_k\to \infty$ as $k\to \infty$. The set of eigenfunctions $\{w_k\}_{k=1}^\infty\subset \mathcal{W}$ forms an orthonormal basis in $\mathcal{H}$, and an orthogonal basis in $\mathcal{V}$. Furthermore, the norm $\|\cdot\|_{\mathcal{W}}$ is equivalent to $\|\mathcal{L}\cdot\|_{L^2(\Omega)}$, see Lemma \ref{W-embedding}. 

Based on the expansion of eigenfunctions, the fractional operator $\L^s$ ($s\geq 0$)  
and the fractional norm $\|\cdot\|_{\mathcal{H}^s}$ are defined as
\begin{equation}\nonumber
   \mathcal{L}^s\psi := \sum_{k=1}^\infty \lambda_k^s a_k w_k\quad \text{ and }\quad  
\|\psi\|_{\mathcal{H}^s}^2 :=\sum_{k=1}^\infty \lambda_k^s|a_k|^2=(\mathcal{L}^\frac{s}2\psi,\mathcal{L}^\frac{s}2\psi), 
\end{equation}
respectively, for any $\psi = \sum\limits_{k=1}^\infty a_k w_k$ with $a_k=(\psi, w_k)$. Then we can define the Hilbert space $\mathcal{H}^s=\mathrm{dom}(\mathcal{L}^\frac{s}{2})$ as the closure of $\mathbb{V}$ in the $\mathcal{H}^s$-norm. 
It follows that 
$$
\mathcal{H}^0=\mathcal{H},\quad \mathcal{H}^1=\mathcal{V},\quad \mathcal{H}^2=\mathcal{W}.
$$ 
As usual, we define extensions of negative exponents as dual spaces, that is, $\mathcal{H}^{-s}=(\mathcal{H}^s)^*$ for $s>0$, with the norm $\|\psi\|_{\mathcal{H}^{-s}}^2 :=\sum_{k=1}^\infty \lambda_k^{-s}|a_k|^2$. Then we have dense inclusions $\mathcal{H}^{s_1}\subset \mathcal{H}^{s_2}$ for $s_1,s_2\in\mathbb{R}$, $s_1\geq s_2$, with continuous embedding when $s_1\geq s_2$ and compact
embedding when $s_1>s_2$. Further properties of the fractional operator $\mathcal{L}^s$ and the fractional space $\mathcal{H}^s$ are summarized in Appendix \ref{sec-appB}.  

The main result of this section is the following 
\begin{theorem}\label{H^s-estimate}
Suppose that $\psi_0 \in \mathcal{H}^s$ for some $s\in (\frac12,1)$. Then problem \eqref{Sharps}--\eqref{bc} admits a global weak solution $(\Bu,p,\psi)$ on $[0,T]$ in the sense of Definition \ref{def-Sharp-weak}, which is unique and satisfies the additional regularity properties 
\[\psi\in L^\infty(0,T;\mathcal{H}^s)\cap L^2(0,T;\mathcal{H}^{1+s})\cap H^1(0,T;\mathcal{H}^{s-1}),\quad \text{ for any }T>0.\]
Let $(\Bu_1, p_1, \psi_1)$ and $(\Bu_2, p_2, \psi_2)$ be two solutions with initial data $\psi_{1,0}$ and $\psi_{2,0}$, respectively. Then the following continuous dependence estimate holds for any $t>0$:
\begin{align}
& \|\L^{\frac{s}{2}}(\psi_1(t)-\psi_2(t))\|_{L^2(\Omega)}^2 \notag \\
&\quad \leq \|\L^{\frac{s}{2}}(\psi_{1,0}-\psi_{2,0})\|_{L^2(\Omega)}^2 \exp\left(\int_0^t C\Big(\|\psi_1(\tau)\|_{\mathcal{H}^s}^\frac{4}{2s-1}+ \| \psi_2(\tau)\|_{\mathcal{H}^s}^\frac{4}{2s-1}+1\Big)\dd \tau\right).\label{S-conti-Hs}
\end{align}
\end{theorem}

From Theorem \ref{H^s-estimate} and the interpolation theory, it follows that $\psi\in C([0,T];\mathcal{H}^s)$ for any $T>0$. Furthermore, we can conclude 
\begin{corollary} 
Let $s\in (\frac12,1)$. 
Problem \eqref{Sharps}--\eqref{bc} generates a strongly continuous semigroup $\{S_0(t)\}_{t\geq 0}$ on the phase space $\mathcal{H}^s$, that is, 
$$
S_0(t): \mathcal{H}^s\to \mathcal{H}^s,\quad S_0(t)\psi_0=\psi(t),\quad \forall\, t\geq 0, 
$$
where $\psi$ is the unique weak solution to problem \eqref{Sharps}--\eqref{bc} corresponding to the initial datum $\psi_0 \in \mathcal{H}^s$.
\end{corollary}

To Theorem \ref{H^s-estimate}, we derive some necessary \textit{a priori} estimates. 
\begin{lemma}[$L^r$-estimate for $\psi$] 
\label{L^r-estimate}
Let $(\Bu,p,\psi)$ be a global weak solution to  problem \eqref{Sharps}--\eqref{bc}. For any given  $r>2$, we assume the constant $\delta$ in \eqref{phi_b1} satisfies 
\begin{equation}\label{delta_1}
0<\delta \leq  \delta_1(r) :=  \left[\frac{\frac{r-1}{r^2} \min_j D_j}{c_\Delta H^\frac{2}{r} \sup_{j}K_j (1+C_p)}\right]^\frac{r}{r-2}.
\end{equation}
If in addition, $\psi_0 \in L^r(\Omega)$, then $\psi\in L^\infty(0,\infty;L^r(\Omega))$ satisfies
\begin{equation}\label{1-1}
\|\psi(t)\|_{L^r(\Omega)}^r \leq \|\psi_0\|_{L^r(\Omega)}^r e^{-\frac{2(r-1)\min_{j}D_j}{rH^2}t} +\frac{rM_0 H^2}{2(r-1)\min_{j}D_j}\left(1- e^{-\frac{2(r-1)\min_{j}D_j}{r H^2}t}\right), 
\end{equation}
for any $t\ge 0$, where the positive constant $M_0$ is given by
\begin{equation}\label{M0}
M _0 := \frac{(8r c_\Delta  L^\frac2r \delta^{-\frac1r}\max_{j}D_j)^r}{r(\min_j D_j)^{r-1}}.
\end{equation}
In particular, we have
\begin{equation}\label{1-2}
\|\psi(t)\|_{L^r(\Omega)}^r\leq \frac{rM_0 H^2}{2(r-1)\min_{j}D_j}+1,
\end{equation}
for any
\begin{equation}
t\ge  T_0 (\|\psi_0\|_{L^r (\Omega)}):= \frac{r  H^2}{2(r-1)\min_j D_j} \ln \big(\|\psi_0\|_{L^r(\Omega)}^r+1\big).
\end{equation}
\end{lemma}
\begin{proof} 
Testing \eqref{Sharps}$_3$ with $r|\psi|^{r-2} \psi$ and using integration by parts, we obtain 
\begin{equation}\label{4-1}
 \frac{\dd}{\dd t}\int_{\Omega} |\psi|^r\dd \Bx +  \frac{4 (r-1)}{r}\int_{\Omega} D|\nabla|\psi|^\frac{r}{2}|^2\dd \Bx + r\int_{\Omega}(\psi_b'\Bu \cdot \Be_z |\psi|^{r-2}\psi - D \varphi_b'' |\psi|^{r-2} \psi)\dd \Bx =0.
\end{equation} 
Recalling that $\psi_b'$ is a smooth function supported in the domain 
$$\Omega_\delta=\Omega\cap (\{-H<z<-H+\delta)\}\cup\{-\delta, 0)\})$$ 
and satisfies \eqref{phi_b2}, by H\"older's inequality and \eqref{Sharps}$_1$, we get 
\begin{equation}\nonumber
\begin{aligned}
\left| r\int_{\Omega}\psi_b'\Bu \cdot \Be_z |\psi|^{r-2} \psi\dd \Bx\right|  
&\leq  rc_\Delta \delta^{-1}\int_{\Omega_\delta} |u_z|  |\psi|^{r-1}\dd \Bx \leq r c_\Delta  \delta^{-1} \|u_z\|_{L^r(\Omega)} \|\psi\|_{L^r(\Omega_\delta)}^{r-1} \\
& \leq rc_\Delta \delta^{-1}\sup_{j}K_j(\|\partial_z p\|_{L^r(\Omega)}+ \|\psi \|_{L^r(\Omega)})\|\psi\|_{L^r(\Omega_\delta)}^{r-1}.
\end{aligned}
\end{equation}  
Then it follows from Lemmas \ref{estimate-p},  \ref{lemmaA5} and the definition of $\Omega_\delta $ that 
\begin{align}
\left|r\int_{\Omega}\psi_b'\Bu \cdot \Be_z |\psi|^{r-2}\psi\dd \Bx\right|  
& \leq rc_\Delta \delta^{-1}\sup_{j}K_j(1+C_p) \|\psi \|_{L^r(\Omega)}\|\psi\|_{L^r(\Omega_\delta)}^{r-1} \notag \\
& \leq rc_\Delta \delta^{-1}\sup_{j}K_j(1+C_p) \||\psi|^\frac{r}{2} \|_{L^2(\Omega)}^\frac{2}{r}\||\psi|^\frac{r}{2}\|_{L^2(\Omega_\delta)}^{\frac{2(r-1)}{r}}\notag \\
& \leq 2r c_\Delta  H^\frac{2}{r}\delta^{\frac{r-2}{r}}\sup_{j}K_j(1+C_p) \|\nabla|\psi|^\frac{r}{2} \|_{L^2(\Omega)}^2. \label{4-2}
\end{align}
Furthermore, by H\"older's inequality, Young's inequality and Lemma \ref{lemmaA5}, we have 
\begin{align}
\left|r\int_\Omega D \varphi_b'' |\psi|^{r-2} \psi \dd \Bx\right| & \leq 2rc_\Delta \delta^{-2}\max_{j}D_j \int_{\Omega_\delta} |\psi|^{r-1}\dd\Bx \notag \\
& \leq 2rc_\Delta \delta^{-2}\max_{j}D_j |\Omega_\delta|^\frac1r\|\psi\|_{L^r(\Omega_\delta)}^{r-1} \notag \\
& \leq 8rc_\Delta  L^\frac2r \delta^{-\frac1r}\max_{j}D_j  \|\nabla(|\psi|^\frac{r}{2}) \|_{L^2(\Omega_\delta)}^{\frac{2(r-1)}{r}} \notag \\
& \leq \frac{r-1}{r}\min_{j} D_j \|\nabla(|\psi|^\frac{r}{2})\|_{L^2(\Omega)}^2+ \frac{(8r c_\Delta   L^\frac2r \delta^{-\frac1r}\max_{j}D_j)^r}{r(\min_j D_j)^{r-1}}.
\label{4-3}
\end{align}
Choosing $\delta$ as in \eqref{delta_1}
and substituting \eqref{4-2}--\eqref{4-3} into \eqref{4-1}, we obtain
\begin{equation}\label{4-4}
 \frac{\dd}{\dd t}\|\psi\|_{L^r(\Omega)}^r + \frac{2(r-1)}{r}\|\sqrt{D}\nabla(|\psi|^\frac{r}{2})\|_{L^2(\Omega)}^2 \leq  \frac{(8r c_\Delta  L^\frac2r \delta^{-\frac1r}\max_{j}D_j)^r}{r(\min_j D_j)^{r-1}}=:M_0. 
\end{equation}
Then from Lemma \ref{lemmaA5} and \eqref{4-4} we infer that 
\begin{equation}\label{4-5}
 \frac{\dd}{\dd t} \|\psi\|_{L^r(\Omega)}^r + \frac{2(r-1)\min_j D_j}{r H^2}\|\psi\|_{L^r(\Omega)}^r\leq M_0,\quad \forall\, t>0.
\end{equation}  
Applying Gronwall's lemma to \eqref{4-5}, we obtain the estimate \eqref{1-1}. The estimate \eqref{1-2} is a direct consequence of \eqref{1-1}. The proof is complete.   
\end{proof}

\begin{lemma}[$\mathcal{H}^s$-estimate for $\psi$]
\label{Hs-es}
Let $s\in (\frac12,1)$. Suppose that $\psi_0 \in \mathcal{H}^s$ and $\delta$ satisfies \eqref{delta_1} with $r = \frac{6}{3-2s}>2$. Then $\psi\in L^\infty(0, T;\mathcal{H}^s)\cap L^2(0,T;\mathcal{H}^{1+s})\cap H^1(0,T;\mathcal{H}^{s-1})$ for any $T>0$ and satisfies 
\begin{align}
  \|\L^{\frac{s}{2}}\psi(t)\|_{L^2(\Omega)}^2 
  & \leq \| \psi_0\|_{\mathcal{H}^s}^2 \exp\left[C \left( \|\psi\|_{L^\infty_t L^r_{\bm{x}}}^\frac{4}{2s-1}+1\right) t\right] \notag \\
  &\quad + C \int_0^t \exp\left[C\left( \|\psi\|_{L^\infty_t L^r_{\bm{x}}}^\frac{4}{2s-1}+1\right) (t-\tau)\right] \dd \tau,\quad \forall\, t>0, \label{H^s-5-1}
\end{align}
where $\|\psi\|_{L^\infty_t L^r_{\bm{x}}}$ is controlled by \eqref{1-1} such that 
\begin{align}
\|\psi\|_{L^\infty_t L^r_{\bm{x}}}^r \leq \|\psi_0\|_{L^r(\Omega)}^r  +\frac{rM_0 H^2}{2(r-1)\min_{j}D_j}. 
\label{uni-t-Lr}
\end{align}
\end{lemma}

\begin{proof} 
Since the space $\mathcal{H}^s$ is continuously embedded into $L^r(\Omega)$ with $r = \frac{6}{3-2s}$, we have $\psi_0\in L^r(\Omega)$. Then by Lemma \ref{L^r-estimate}, we see that $\psi\in L^\infty(0,\infty;L^r(\Omega))$ satisfying the uniform in time estimate \eqref{uni-t-Lr}. 
Next, testing equation \eqref{Sharps}$_3$ by $\mathcal{L}^{s} \psi$, after integration by parts, we obtain
\begin{equation}\label{H^s-1}
\frac{1}{2}\frac{\dd}{\dd t}\|\L^{\frac{s}{2}}\psi\|_{L^2(\Omega)}^2 + \|\mathcal{L}^{\frac{1+s}{2}}\psi\|_{L^2(\Omega)}^2= -(\Bu\cdot\nabla \psi, \mathcal{L}^{s}\psi) - (\phi_b'\Bu \cdot \Be_z, \mathcal{L}^{s}\psi) + ( D\varphi_b'', \mathcal{L}^{s}\psi).
\end{equation}
Using H\"older's inequality, Lemma \ref{estimate-p} and \eqref{Sharps}$_1$, we have
\begin{equation}\nonumber
\begin{aligned}
|(\Bu\cdot\nabla \psi, \mathcal{L}^{s}\psi)| 
&\leq \|\Bu\|_{L^r(\Omega)}\|\nabla \psi\|_{L^{3}(\Omega)}\|\mathcal{L}^{s}\psi\|_{L^\frac{6}{1+2s}(\Omega)}\\
&\leq C(\|\nabla p\|_{L^r(\Omega)}+\|\psi\|_{L^r(\Omega)})\|\nabla \psi\|_{L^{3}(\Omega)}\|\mathcal{L}^{s}\psi\|_{L^\frac{6}{1+2s}(\Omega)}\\
&\leq C \|\psi\|_{L^r(\Omega)}\|\nabla \psi\|_{L^{3}(\Omega)}\|\mathcal{L}^{s}\psi\|_{L^\frac{6}{1+2s}(\Omega)}.
\end{aligned}
\end{equation}
By the Sobolev embedding theorem and Lemmas \ref{lem-B1}--\ref{lem-B3}, we find
\begin{equation}\nonumber
\begin{aligned}
\|\nabla \psi\|_{L^{3}(\Omega)} 
& \leq C\sum_{j=1}^l\|\nabla \psi\|_{L^{3}(\Omega_j)}  \leq C\sum_{j=1}^l\| \psi\|_{H^{\frac32}(\Omega_j)} \leq C\|\psi\|_{\mathcal{H}^\frac32}\\
& = C\|\mathcal{L}^\frac34\psi\|_{L^2(\Omega)} \leq C\|\mathcal{L}^\frac{s}{2}\psi\|_{L^2(\Omega)}^{s-\frac12}\|\mathcal{L}^\frac{1+s}{2}\psi\|_{L^2(\Omega)}^{\frac32 -s}
\end{aligned}
\end{equation} 
and 
\begin{equation}\nonumber
\begin{aligned}
\|\mathcal{L}^{s}\psi\|_{L^\frac{6}{1+2s}(\Omega)} \leq C\|\mathcal{L}^{s}\psi\|_{H^{1-s}(\Omega)}\leq C\|\mathcal{L}^{\frac{1+s}{2}}\psi\|_{L^2(\Omega)}.
\end{aligned}
\end{equation}
Then it follows from Young's inequality that  
\begin{align}
|(\Bu\cdot\nabla \psi, \mathcal{L}^{s}\psi)| 
& \leq C \|\psi\|_{L^r(\Omega)}\|\mathcal{L}^\frac{s}{2}\psi\|_{L^2(\Omega)}^{s-\frac12}\|\mathcal{L}^{\frac{1+s}{2}}\psi\|_{L^2(\Omega)}^{\frac52-s} \notag \\
& \leq \frac16\|\mathcal{L}^{\frac{1+s}{2}}\psi\|_{L^2(\Omega)}^2+ C\|\psi\|_{L^r(\Omega)}^\frac{4}{2s-1} \|\mathcal{L}^{\frac{s}{2}}\psi\|_{L^2(\Omega)}^2. 
\label{H^s-2}
\end{align}
Furthermore, by virtue of \eqref{Sharps}$_1$ and Lemma \ref{lem-B1}, we get  
\begin{align}
|(\phi_b'\Bu \cdot \Be_z, \mathcal{L}^{s}\psi) | 
& \leq C\|\Bu\|_{L^2(\Omega)} \| \mathcal{L}^{s}\psi \|_{L^2(\Omega)} \leq C\|\psi\|_{L^2(\Omega)} \| \mathcal{L}^{\frac{1+s}{2}}\psi \|_{L^2(\Omega)} \notag \\
& \leq C\|\mathcal{L}^\frac{s}{2}\psi\|_{L^2(\Omega)} \| \mathcal{L}^{\frac{1+s}{2}}\psi \|_{L^2(\Omega)} \notag \\
& \leq
\frac16\|\mathcal{L}^{\frac{1+s}{2}}\psi\|_{L^2(\Omega)}^2 + C\|\mathcal{L}^\frac{s}{2}\psi\|_{L^2(\Omega)}^2
\label{H^s-3}
\end{align}
and 
\begin{equation}\label{H^s-4}
\begin{aligned}
|( D\varphi_b'', \mathcal{L}^{s}\psi)| & \leq   \|D\varphi_b''\|_{L^2(\Omega)}\|\mathcal{L}^{s}\psi\|_{L^2(\Omega)} \leq C\|\mathcal{L}^{\frac{1+s}{2}}\psi\|_{L^2(\Omega)} \leq \frac16\|\mathcal{L}^{\frac{1+s}{2}}\psi\|_{L^2(\Omega)}^2 + C.
\end{aligned}
\end{equation}
Substituting \eqref{H^s-2}--\eqref{H^s-4} into \eqref{H^s-1}, we obtain
\begin{equation}\label{H^s-5}
 \frac{\dd}{\dd t}\|\L^{\frac{s}{2}}\psi\|_{L^2(\Omega)}^2 + \|\mathcal{L}^{\frac{1+s}{2}}\psi\|_{L^2(\Omega)}^2 
 \leq C\left(\|\psi\|_{L^r(\Omega)}^\frac{4}{2s-1} +1\right) \|\mathcal{L}^\frac{s}{2}\psi\|_{L^2(\Omega)}^2 + C.
\end{equation}
From \eqref{H^s-5}, Gronwall's lemma and \eqref{uni-t-Lr}, we infer that $\psi\in L^\infty(0, T;\mathcal{H}^s)\cap L^2(0,T;\mathcal{H}^{1+s})$ for any $T>0$ and satisfies the estimate \eqref{H^s-5-1}. Finally, by comparison in \eqref{Sharps}$_3$, we get  
\begin{align*}
\|\partial_t \psi\|_{\mathcal{H}^{s-1}} 
&\leq  \| \Bu\cdot\nabla \psi\|_{\mathcal{H}^{s-1}}  +  \|\phi_b'\Bu \cdot \Be_z \|_{\mathcal{H}^{s-1}}+ \| \mathcal{L}\psi\|_{\mathcal{H}^{s-1}}+\|D \varphi_b''\|_{\mathcal{H}^{s-1}}\\
&\leq  C\left( \| \Bu\cdot\nabla \psi\|_{L^{\frac{6}{5-2s}}(\Omega)}  +  \|\phi_b'\Bu \cdot \Be_z \|_{L^{\frac{6}{5-2s}}(\Omega)} + \| \psi\|_{\mathcal{H}^{1+s}}+\|D \varphi_b''\|_{L^{\frac{6}{5-2s}}(\Omega)}\right)\\
&\leq C\|\psi\|_{L^r(\Omega)}\|\nabla \psi\|_{L^3(\Omega)} + C\| \psi\|_{\mathcal{H}^{1+s}} +C\\
&\leq C\|\psi\|_{\mathcal{H}^s}^{\frac12+s} \|\psi\|_{\mathcal{H}^{1+s}}^{\frac32-s}
+ C\| \psi\|_{\mathcal{H}^{1+s}} +C,
\end{align*}
which implies that $\partial_t\psi\in L^2(0,T;\mathcal{H}^{s-1})$. The proof is complete. 
\end{proof}

Now we are in a position to prove Theorem \ref{H^s-estimate}.
\begin{proof}[Proof of Theorem \ref{H^s-estimate}] 
The existence of a global solution with required regularity properties can be established by the Faedo--Galerkin approximation scheme described in \cite{CNW2025}, the \textit{a priori} estimates in Lemmas \ref{L^r-estimate}, \ref{Hs-es} together with the compactness argument. Hence, we omit the details. 

Next, we show the continuous dependence of the solution on the initial datum. Let $(\Bu_1, p_1, \psi_1)$ and $(\Bu_2, p_2, \psi_2)$ be two solutions corresponding to the initial data $\psi_{1,0}$ and $\psi_{2,0}$, respectively. Then their difference 
$(\overline{\Bu}, \overline{\psi}, \overline{p}) : = (\Bu_1 - \Bu_2, \psi_1 - \psi_2, p_1 - p_2)$ satisfies
\begin{equation}\label{H^s-6}
\left\{\begin{aligned}
&\overline{\Bu} = -K(\nabla \overline{p} + \overline{\psi} \Be_z),\\
&\nabla \cdot \overline{\Bu} = 0,\\
&\partial_t \overline{\psi} +\Bu_1\cdot\nabla \overline{\psi} + \overline{\Bu}\cdot\nabla  \psi_2 + \phi_b'\overline{\Bu}\cdot \Be_z - \operatorname{div}(D\nabla\overline{\psi}) = 0,
\end{aligned}\right.
\end{equation}
and 
\begin{equation}\label{H^s-6-ini}
 \overline{\psi}|_{t=0} =   \overline{\psi}_{0}:= \psi_{1,0} -\psi_{2,0}.
\end{equation}
Testing \eqref{H^s-6}$_3$ by $\mathcal{L}^{s}\overline{\psi}$ and using integration by parts, we obtain
\begin{equation}\label{H^s-7}
\frac{1}{2}\frac{\dd}{\dd t}\|\L^{\frac{s}{2}}\overline{\psi}\|_{L^2(\Omega)}^2 + \|\mathcal{L}^{\frac{1+s}{2}}\overline{\psi}\|_{L^2(\Omega)}^2= -(\Bu_1\cdot\nabla \overline{\psi}, \L^s\overline{\psi})-(\overline{\Bu}\cdot\nabla \psi_2, \L^s\overline{\psi}) - (\phi_b'\overline{\Bu} \cdot \Be_z, \L^s\overline{\psi}).
\end{equation}
Similar to \eqref{H^s-2}--\eqref{H^s-3}, we can derive the following inequality
\begin{equation}\nonumber
\begin{aligned}
 & \frac{\dd}{\dd t}\|\L^{\frac{s}{2}}\overline{\psi}\|_{L^2(\Omega)}^2 + \frac12\|\mathcal{L}^{\frac{1+s}{2}}\overline{\psi}\|_{L^2(\Omega)}^2 \\
 &\quad \leq C\|\psi_1 \|_{L^r(\Omega)}^\frac{4}{2s-1} \|\mathcal{L}^{\frac{s}{2}}\overline{\psi}\|_{L^2(\Omega)}^2  + C\|\overline{\psi}\|_{L^r(\Omega)}^\frac{4}{2s-1} \|\mathcal{L}^{\frac{s}{2}}\psi_2\|_{L^2(\Omega)}^2  + C\|\mathcal{L}^\frac{s}{2}\overline{\psi}\|_{L^2(\Omega)}^2,
 \end{aligned}
\end{equation}
where $r = \frac{6}{3-2s}$. Using the Sobolev embedding theorem, we have
\begin{align}\nonumber
& \|\psi_1 \|_{L^r(\Omega)} 
\leq C\|\L^\frac{s}{2}\psi_1 \|_{L^2(\Omega)},\quad \|\overline{\psi}\|_{L^r(\Omega)}  
\leq C\|\L^\frac{s}{2}\overline{\psi}\|_{L^2(\Omega)}.
\end{align}
Then it follows that 
\begin{align}
& \frac{\dd}{\dd t}\|\L^{\frac{s}{2}}\overline{\psi}\|_{L^2(\Omega)}^2 +\|\mathcal{L}^s\overline{\psi}\|_{L^2(\Omega)}^2
\notag  \\
& \quad \leq C\|\L^\frac{s}{2}\psi_1\|_{L^2(\Omega)}^\frac{4}{2s-1} \|\mathcal{L}^{\frac{s}{2}}\overline{\psi}\|_{L^2(\Omega)}^2  + C\|\L^\frac{s}{2}\overline{\psi}\|_{L^2(\Omega)}^\frac{4}{2s-1} \|\mathcal{L}^{\frac{s}{2}}\psi_2\|_{L^2(\Omega)}^2  + C\|\mathcal{L}^\frac{s}{2}\overline{\psi}\|_{L^2(\Omega)}^2 
\notag \\
& \quad \leq C\|\L^\frac{s}{2}\overline{\psi}\|_{L^2(\Omega)}^2 \left( \|\mathcal{L}^{\frac{s}{2}}\psi_1\|_{L^2(\Omega)}+ \|\mathcal{L}^{\frac{s}{2}}\psi_2\|_{L^2(\Omega)}  \right)^ { \frac{6-4s}{2s-1} }\|\mathcal{L}^{\frac{s}{2}}\psi_2\|_{L^2(\Omega)}^2 
\notag \\
&\qquad + C\|\mathcal{L}^\frac{s}{2}\overline{\psi}\|_{L^2(\Omega)}^2 (\|\L^\frac{s}{2}\psi_1\|_{L^2(\Omega)}^\frac{4}{2s-1}+1) 
\notag \\
&\quad \leq C\|\mathcal{L}^\frac{s}{2}\overline{\psi}\|_{L^2(\Omega)}^2 \left(\|\psi_1\|_{\mathcal{H}^s}^\frac{4}{2s-1} 
+ \| \psi_2\|_{\mathcal{H}^s}^\frac{4}{2s-1}+1\right).
\label{H^s-8}
\end{align}
An application of Gronwall's lemma to \eqref{H^s-8} gives the continuous dependence estimate \eqref{S-conti-Hs}. This together with the fact that $\psi_i\in L^\infty(0,T;\mathcal{H}^s)$ for any $T>0$ and $i=1,2$, yields the continuous dependence of $\psi(t)$ on its initial datum in $\mathcal{H}^s$. In particular, if $\psi_{1,0} =\psi_{2,0}$, then $\psi_1(t)=\psi_2(t)$ for all $t\geq 0$. Hence, the solution is unique. 
\end{proof}

Concerning the diffuse interface model \eqref{Diffuse-s}--\eqref{bc-epsilon} with continuous coefficients $K^\varepsilon$, $D^\varepsilon$, we introduce analogously the higher-order space 
\begin{equation}\nonumber
\mathcal{W}^\varepsilon = \{\psi \in \mathcal{V}:~\partial_x \psi,\,\partial_y \psi \in H^1(\Omega),~D^\varepsilon\partial_z \psi \in H^1(\Omega)\}
\end{equation}
endowed with the following norm    
\[\|\psi\|_{\mathcal{W}^\varepsilon}^2  =\|\psi\|_{H^1(\Omega)}^2 
+ \|\partial_x \psi\|_{H^1(\Omega)}^2
+ \|\partial_y \psi\|_{H^1(\Omega)}^2
+ \|D^\varepsilon\partial_z \psi\|_{H^1(\Omega)}^2.\]
By the construction of $D^\varepsilon$, $\|(D^\varepsilon)'\|_{L^\infty(\Omega)}$ is bounded for any fixed $0<\varepsilon\ll 1$. Using this fact, we can verify that $\mathcal{W}^\varepsilon=\mathcal{V}\cap H^2(\Omega)$. Next, we define the self-adjoint positive operator
$$
\L^\varepsilon \psi := -\div(D^\varepsilon \nabla\psi),\quad \forall\,\psi\in \mathcal{W}^\varepsilon,
$$ 
which admits properties similar to those of $\mathcal{L}$ (see Appendix \ref{sec-appB}). Moreover, the fractional spaces $\mathcal{H}_\varepsilon^s= \mathrm{dom}((\mathcal{L}^\varepsilon)^\frac{s}{2})$ are defined as the closure of $\mathbb{V}$ in the $\mathcal{H}_\varepsilon^s$-norm ($s\in \mathbb{R}$). 
Then we have 
$$
\mathcal{H}_\varepsilon^0=\mathcal{H},\quad  \mathcal{H}^1_\varepsilon=\mathcal{V}, \quad  \mathcal{H}^2_\varepsilon=\mathcal{W}^\varepsilon.
$$
Similarly to Lemma \ref{lem-B2}, the fractional $\mathcal{H}_\varepsilon^s$-norm ($0<s<1$) can be interpreted as the interpolation norm between $\mathcal{H}$ and $\mathcal{V}$, and is equivalent to the fractional Sobolev norm $\|\cdot\|_{H^s(\Omega)}$. For any $0<\varepsilon\ll 1$, we can identify $\mathcal{H}_\varepsilon^s=\mathcal{H}^s$ for $0<s<1$.

Analogous results can be obtained for the diffuse interface model \eqref{Diffuse-s}--\eqref{bc-epsilon} using similar arguments for the sharp interface model. In particular, we have the following \textit{a priori} $L^r$-estimate: 

\begin{lemma}\label{L^r-estimate-diffuse}
Let  $(\Bu^\varepsilon, p^\varepsilon, \psi^\varepsilon)$ be a global weak solution to problem \eqref{Diffuse-s}--\eqref{bc-epsilon}. For any given $r>2$, we assume that the constant $\delta$ in \eqref{phi_b1} satisfies \eqref{delta_1}. If $\psi_0^\varepsilon \in L^r(\Omega)$, then $\psi\in L^\infty(0,\infty;L^r(\Omega))$ satisfies
\begin{equation}\label{L^r-estimate-diffuse-1}
\|\psi^\varepsilon(t)\|_{L^r(\Omega)}^r \leq \|\psi_0^\varepsilon\|_{L^r(\Omega)}^r e^{-\frac{2(r-1)\min_{j}D_j}{rH^2}t} +\frac{rM_0 H^2}{2(r-1)\min_{j}D_j}\left(1- e^{-\frac{2(r-1)\min_{j}D_j}{r H^2}t}\right), 
\end{equation}
for any $t\ge 0$, where the constant $M_0$ is given by \eqref{M0} and is independent of $\varepsilon$. In particular, we have 
\begin{equation}\label{L^r-estimate-diffuse-2}
\|\psi^\varepsilon(t)\|_{L^r(\Omega)}^r\leq \frac{rM_0 H^2}{2(r-1)\min_{j}D_j}+1 
\end{equation}
for any  
\begin{equation}\nonumber
t\ge  T_0 (\|\psi_0^\varepsilon\|_{L^r (\Omega)})= \frac{r  H^2}{2(r-1)\min_j D_j} \ln (\|\psi_0^\varepsilon\|_{L^r(\Omega)}^r+1).
\end{equation}
\end{lemma}

With the aid of Lemma \ref{L^r-estimate-diffuse}, we can prove the well-posedness of the diffuse interface model \eqref{Diffuse-s}--\eqref{bc-epsilon} in $\mathcal{H}^s$ with $s\in (\frac{1}{2},1)$. 
\begin{theorem}\label{prop1-1}
Suppose that $0<\varepsilon \ll 1$. For any initial data $\psi_0^\varepsilon \in L^2(\Omega)$, problem \eqref{Diffuse-s}--\eqref{bc-epsilon} admits a global weak solution $(\Bu,p,\psi)$ such that $\psi^\varepsilon \in L^\infty(0,T;L^2(\Omega)) \cap L^2(0,T;\mathcal{V})\cap H^1(0,T;\mathcal{V}^*)$, $p^\varepsilon\in L^2(0,T;\mathcal{V})$ and $\bm{u}^\varepsilon\in L^2(0,T;\bm{H})$ for any $T>0$. If, in addition, $\psi^\varepsilon_0\in \mathcal{H}^s$ for some $s\in (\frac12,1)$, then the solution is unique and satisfies 
\[\psi^\varepsilon\in L^\infty(0,T;\mathcal{H}^s)\cap L^2(0,T;\mathcal{H}_\varepsilon^{1+s})\cap H^1(0,T;\mathcal{H}_\varepsilon^{s-1}),~~~~\text{ for any }T>0,\]
which also implies $\psi^\varepsilon\in C([0,T];\mathcal{H}^s)$.  
Let $(\Bu_1^\varepsilon, \psi_1^\varepsilon, p_1^\varepsilon)$ and $(\Bu_2^\varepsilon, \psi_2^\varepsilon, p_2^\varepsilon)$ be two solutions corresponding to the initial data $\psi_{1,0}^\varepsilon$ and $\psi_{2,0}^\varepsilon$, respectively. Then the following continuous dependence estimate holds for any $t>0$:
\begin{align}
& \|(\L^\varepsilon)^{\frac{s}{2}}(\psi_1^\varepsilon(t)-\psi_2^\varepsilon(t))\|_{L^2(\Omega)}^2 
\notag \\
&\quad \leq \|(\L^\varepsilon)^{\frac{s}{2}}(\psi^\varepsilon_{1,0}-\psi_{2,0}^\varepsilon)\|_{L^2(\Omega)}^2 \exp\left(\int_0^t C\Big(\|\psi_1^\varepsilon(\tau)\|_{\mathcal{H}^s}^\frac{4}{2s-1}+ \| \psi_2^\varepsilon(\tau)\|_{\mathcal{H}^s}^\frac{4}{2s-1}+1\Big)\dd \tau\right).
\label{D-conti-Hs}
\end{align}
\end{theorem}
As a consequence of Theorem \ref{prop1-1}, we have  
\begin{corollary} 
Let $s\in (\frac12,1)$ and $0<\varepsilon\ll 1$. Problem \eqref{Diffuse-s}--\eqref{bc-epsilon} generates a strongly continuous semigroup $\{S_\varepsilon(t)\}_{t\geq 0}$ on the phase space $\mathcal{H}^s$, that is, 
$$
S_\varepsilon(t): \mathcal{H}^s\to \mathcal{H}^s,\quad S_\varepsilon(t)\psi_0^\varepsilon=\psi^\varepsilon(t),\quad \forall\, t\geq 0, 
$$
where $\psi^\varepsilon$ is the unique weak solution to problem \eqref{Diffuse-s}--\eqref{bc-epsilon}  corresponding to the initial datum $\psi_0^\varepsilon \in \mathcal{H}^s$.
\end{corollary}

\section{Existence of Global Attractors}
In this section, we study the long-time behavior of global weak solutions to the sharp interface model \eqref{Sharps}--\eqref{bc} and the diffuse interface model \eqref{Diffuse-s}--\eqref{bc-epsilon} in the framework of infinite-dimensional dynamical systems. More precisely, we prove the existence of global attractors for both models in the phase space $\mathcal{H}^{s}$ with $s\in (\frac12,1)$. 

The following lemma (see \cite{Temam1997}) plays a crucial role in deriving uniform-in-time estimates.
\begin{lemma}[Uniform Gronwall lemma]
\label{lemmaUGL}
Let $g, h, y$ be three positive locally integrable functions on $(t_0,\infty)$ such that $y'$ is locally integrable on
$(t_0,\infty)$, and which satisfy
\begin{equation}
  y' \leq g y + h,\quad \forall\, t>t_0, \notag 
\end{equation} 
\begin{equation}
  \int_{t}^{t+r} g(s) \dd s \leq  a_1, ~~~\int_{t}^{t+r} h(s) \dd s \leq  a_2, ~~~\int_{t}^{t+r} y(s) \dd s \leq  a_3, \quad \forall\, t>t_0. \notag 
\end{equation} 
where $r,a_1,a_2,a_3$ are positive constants. Then
\begin{equation}
y(t+r) \leq \left(\frac{a_3}{r}+a_2\right) e^{a_1},\quad \forall\, t\geq  t_0.\notag 
\end{equation}
\end{lemma}

\subsection{Sharp interface model} 
We first study the sharp interface model \eqref{Sharps}--\eqref{bc}. 

\begin{theorem}\label{attractor-sharp}
Let $s\in (\frac12,1)$. The semigroup $\{S_0(t)\}_{t\geq  0}$ associated with the sharp interface model \eqref{Sharps}--\eqref{bc}  possesses a global attractor $\mathscr{A}_0 \subset \mathcal{H}^s$, which is bounded in $\mathcal{V}$, compact and connected in $\mathcal{H}^s$. 
\end{theorem}
The following lemma yields the existence of an absorbing set in $\mathcal{H}$.
\begin{lemma}[Dissipative $L^2$-estimate]
\label{prop2}
Let $\psi_0\in \mathcal{H}$ and the constant $\delta$ in \eqref{phi_b1} satisfy
\begin{equation}
\label{delta_2}
0<\delta \leq \delta_2 :=   \frac{\min_{j} K_j \min_{j}D_j}{32(\max_{j} K_j)^2 c_\Delta}.
\end{equation} 
%
Then the solution $\psi$ of problem \eqref{Sharps}--\eqref{bc} satisfies
\begin{equation}\label{2-2}
\|\psi(t)\|_{L^2(\Omega)}^2 \leq \|\psi_0\|_{L^2(\Omega)}^2 e^{-\frac{\min_{j}D_j}{H^2}t} +\frac{M_1H^2}{\min_{j}D_j}\left(1- e^{-\frac{\min_{j}D_j}{H^2}t}\right),\quad \forall\,t\geq  0,
\end{equation}
where 
\begin{equation}\label{M1-delta}
M_1 = \frac{8 c_\Delta^2 L^2}{\delta}\frac{ (\max_j D_j )^2}{\min_{j} D_j }.
\end{equation}
In addition, for any  
$$t\ge T_1:=T_1(\|\psi_0\|_{L^2(\Omega)}) =\frac{H^2}{\min_{j}D_j}\ln (\|\psi_0\|_{L^2(\Omega)}^2+1),$$ it holds
\begin{equation}\label{2-3}
\|\psi(t) \|_{L^2(\Omega)}^2\leq  \frac{M_1H^2}{\min_{j}D_j}+1,\quad \int_t^{t+1} \|\sqrt{D}\nabla\psi(\tau) \|_{L^2(\Omega)}^2 \dd \tau \leq  \frac{M_1H^2 }{\min_{j}D_j } +1 +M_1.
\end{equation}
\end{lemma}
\begin{proof}
Testing \eqref{Sharps}$_3$ by $\psi$ and using integration by parts, we get 
\begin{equation}\nonumber
\frac12  \frac{\dd}{\dd t} \|\psi\|_{L^2(\Omega)}^2 + \|\sqrt{D}\nabla\psi\|_{L^2(\Omega)}^2 =  -   \int_\Omega  \phi_b' u_z \psi \dd \Bx + \int_\Omega D \varphi_b''\psi \dd \Bx .
\end{equation} 
By virtue of \eqref{phi_b1} and \eqref{phi_b2},  it holds 
\begin{equation}\label{2-3-0}
\left|\int_\Omega  \phi_b' u_z \psi \dd \Bx\right| = \left|\int_{\Omega_\delta}  \phi_b' u_z \psi \dd \Bx\right| \leq  c_\Delta \delta^{-1}\|u_z\|_{L^2(\Omega_\delta)} \|\psi\|_{L^2(\Omega_\delta)},
\end{equation} 
where $\Omega_\delta := \Omega \cap (\{-H<z<-H+\delta\} \cup \{-\delta <z<0\})$. Noting that both $u_z$ and $\psi$ vanish on  the boundary $\{z=0\}\cup \{z=-H\}$, we infer from Lemma \ref{lemmaA5} that 
\begin{equation}\label{2-3-1}
\|u_z\|_{L^2(\Omega_\delta)} \leq 2  \delta\|\partial_z u_z\|_{L^2(\Omega_\delta)},\quad \|\psi\|_{L^2(\Omega_\delta)} \leq  2 \delta\|\partial_z \psi\|_{L^2(\Omega_\delta)}.
\end{equation}
As a consequence, it holds 
\begin{align} 
\left|\int_\Omega  \phi_b' u_z \psi \dd \Bx\right| 
& \leq 4c_\Delta\delta\|\partial_z u_z\|_{L^2(\Omega_\delta)}\|\partial_z \psi\|_{L^2(\Omega_\delta)} \notag \\
& \leq  4c_\Delta\delta (\|\partial_x u_x\|_{L^2(\Omega)}+ \|\partial_y u_y\|_{L^2(\Omega)})\|\partial_z \psi\|_{L^2(\Omega)}.
\label{2-4}
\end{align}
Here, we use the incompressibility condition $\nabla \cdot \Bu =0$. In view of \eqref{Sharps}$_1$, we have
\begin{equation}\label{2-5}
\|\partial_x \Bu\|_{L^{2}(\Omega)} \leq \max_{j} K_j(\|\partial_x \nabla p\|_{L^{2}(\Omega)}+ \|\partial_x  \psi\|_{L^2(\Omega)}),
\end{equation}
where the pressure $p$ is determined by the elliptic problem \eqref{pressure}. Since the permeability coefficient $K$ is independent of $x$, one can take the derivative of \eqref{pressure} with respect to $x$ and obtain 
 \begin{equation}\nonumber\label{pressure_x}
  \left\{\begin{aligned}
&-\div(K\nabla  \partial_x p) = \div (K\partial_x \psi \Be_z) ,  ~~~~\text{ in } \Omega,\\
&\partial_z \partial_x p(x,-H)  = \partial_z \partial_x p(x,0)  = 0,
  \end{aligned}\right. 
\end{equation}
and $\partial_x p$ is periodic in the horizontal directions $(x,y)$. From the standard elliptic estimate, we get
\begin{equation}\nonumber
\|\partial_x\nabla p\|_{L^2(\Omega)} \leq   \frac{\max_{j} K_j}{\min_{j} K_j}\|\partial_x\psi\|_{L^2(\Omega)}.
\end{equation}
This, together with \eqref{2-5}, yields
\begin{equation}\nonumber
\|\partial_x \Bu\|_{L^{2}(\Omega)} \leq  \frac{2(\max_{j} K_j)^2}{\min_{j} K_j} \|\partial_x\psi\|_{L^2(\Omega)}.
\end{equation}
Similarly, we have  
\begin{equation}\nonumber
\|\partial_y \Bu\|_{L^{2}(\Omega)} \leq  \frac{2(\max_{j} K_j)^2}{\min_{j} K_j} \|\partial_y\psi\|_{L^2(\Omega)}.
\end{equation}
As a consequence, 
\begin{equation}\label{2-6}
\left|\int_\Omega  \phi_b' u_z \psi \dd \Bx \right|\leq   \frac{8(\max_{j} K_j)^2 c_\Delta \delta }{\min_{j} K_j} \|\nabla \psi\|_{L^2(\Omega)}^2.
\end{equation}
On the other hand, using H\"older's inequality, Young's inequality and \eqref{2-3-1}, we find
\begin{align}
\left|\int_\Omega  D \varphi_b''\psi \dd \Bx\right| & \leq 2c_\Delta \delta^{-2} \max_j D_j  |\Omega_\delta|^\frac12 \| \psi\|_{L^2(\Omega_\delta)}  \notag \\
& \leq 4c_\Delta L  \delta^{-\frac12} \max_j D_j    \| \partial_z\psi\|_{L^2(\Omega_\delta)} \notag \\ 
& \leq \frac{1}{4}\min_{j} D_j \|\nabla \psi\|_{L^2(\Omega)}^2+  \frac{16 c_\Delta^2 L^{2}}{\delta}\frac{ (\max_j D_j )^2}{\min_{j} D_j }.
\label{2-7}
\end{align}
Substituting \eqref{2-6}, \eqref{2-7} into \eqref{2-4} and choosing $\delta$ satisfying \eqref{delta_2},
we get  
\begin{equation}\label{2-8}
 \frac{\dd}{\dd t} \|\psi\|_{L^2(\Omega)}^2  +   \|\sqrt{D}\nabla\psi\|_{L^2(\Omega)}^2 \leq  \frac{32 c_\Delta^2 L^2}{\delta}\frac{ (\max_j D_j )^2}{\min_{j} D_j }=:M_1.
\end{equation}  
Due to Lemma \ref{lemmaA5}, we also infer from \eqref{2-8} that 
\begin{equation}\nonumber
 \frac{\dd}{\dd t} \|\psi\|_{L^2(\Omega)}^2  + \frac{\min_{j}D_j}{H^2} \|\psi\|_{L^2(\Omega)}^2 \leq M_1.
\end{equation}  
Applying the classical Gronwall lemma, we obtain the estimate \eqref{2-2}. In particular, it follows from \eqref{2-2} that
\begin{equation}\label{2-10}
\|\psi(t)\|_{L^2(\Omega)}^2 \leq \frac{M_1H^2}{\min_{j}D_j}+1,\quad \forall\,t \ge T_1=\frac{H^2}{\min_{j}D_j}\ln  (\|\psi_0\|_{L^2(\Omega)}^2+1).
\end{equation}
In addition, for $t\ge T_1$, we integrate the energy inequality \eqref{2-8} over $[t,t+1]$ to obtain
\begin{equation}\label{2-11}
\|\psi(t+1) \|_{L^2(\Omega)}^2+  \int_t^{t+1} \|\sqrt{D}\nabla\psi(s) \|_{L^2(\Omega)}^2 \dd s \leq \|\psi(t) \|_{L^2(\Omega)}^2 + M_1 \leq \frac{M_1H^2}{\min_{j}D_j}+1+M_1.
\end{equation}
Combining \eqref{2-10} and \eqref{2-11}, we arrive at the conclusion \eqref{2-3}.
\end{proof}

The next lemma gives a uniform estimate of $\psi$ in $\mathcal{V}$ for large time, which implies that the trajectory $\psi(t)$ is uniformly compact in $\mathcal{H}^s$ with $s\in(\frac12,1)$. 
\begin{lemma}\label{prop3}
Let $s\in (\frac12,1)$ and $\psi_0\in \mathcal{H}^s$. We take $r=\frac{6}{3-2s}$ and  
\begin{equation}\label{def-delta}
\delta = \min\{\delta_1,\delta_2\},
\end{equation}
where $\delta_1$ and $\delta_2$ are given in \eqref{delta_1} and \eqref{delta_2}, respectively. 
For any $\tau\in (0,1)$, the solution $\psi$ of problem \eqref{Sharps}--\eqref{bc} satisfies $\psi\in L^\infty(\tau, \infty;\mathcal{V})$. Moreover, for any $t\ge \max\{ T_0,T_1\}+1$, where $T_0$, $T_1$ are given in Lemma \ref{L^r-estimate}, Lemma \ref{prop2}, respectively, 
it holds 
\begin{align}
\|\sqrt{D}\nabla \psi(t)\|_{L^2(\Omega)}^2  
&\leq \left(M_2+ \frac{M_1H^2}{\min_{j}D_j}+1+M_1\right) \operatorname{exp}\left[M_3+M_4 \left( \frac{rM_0 H^2}{2(r-1)\min_{j}D_j}+1\right)^\frac{2}{r-3}\right] 
\notag \\
& =:M_5. \label{2-14}
\end{align}
with
\begin{equation}\label{M2M3}
  M_2= 32 c_\Delta^2 L^2 \delta^{-3} (\max_j D_j )^2,\quad M_3 = \frac{32(\max_{j} K_j)^4 c_\Delta^2  }{(\min_{j} K_j)^2\min_{j} D_j},
\end{equation}
and 
\begin{equation}\label{M4}
M_4 = \frac{2 C_s\big[(1+C_p) C_2 C_u^\frac{3}{r}\max_j K_j\big]^\frac{2r}{r-3}}{\min_j D_j}.
\end{equation} 
\end{lemma}
\begin{proof}
It follows from Theorem \ref{H^s-estimate} that $\psi \in L^2(0,T;\mathcal{H}^{1+s}) \subset L^2(0,T;\mathcal{V})$ for any $T\geq  1$. Then for any $\tau\in (0,1)$ there exists some $\tau'\in (0,\tau)$ such that $\psi(\tau') \in \mathcal{V}$. By uniqueness, $\psi$ can be viewed as a solution on $[\tau',T]$ with the initial datum $\psi|_{t= \tau'} = \psi(\tau')$. From Proposition \ref{prop1}, we have $\psi \in L^\infty(\tau', T;\mathcal{V})$. 

In what follows, we show that $\psi \in L^\infty(\tau', \infty;\mathcal{V})$. Multiplying \eqref{Sharps}$_3$ by $\mathcal{L}\psi$ and integrating over $\Omega$, we obtain
\begin{equation}\label{2-15}
\frac12  \frac{\dd}{\dd t} \|\sqrt{D}\nabla \psi\|_{L^2(\Omega)}^2  +\|\mathcal{L}\psi\|_{L^2(\Omega)}^2  = \int_\Omega \Bu\cdot\nabla \psi \mathcal{L}\psi \dd \Bx+  \int_\Omega  \phi_b' u_z \mathcal{L}\psi \dd \Bx  - \int_\Omega D \varphi_b''\mathcal{L}\psi \dd\Bx .
\end{equation}
Using H\"older's inequality, Young's inequality and \eqref{2-3-1}, we get 
\begin{align}
\left|\int_\Omega  D \varphi_b''\mathcal{L}\psi \dd \Bx \right| 
& \leq  2c_\Delta\delta^{-2} \max_j D_j |\Omega_\delta|^\frac12 \|\mathcal{L}\psi\|_{L^2(\Omega_\delta)}  
\notag \\
& \leq \frac{1}{4}\|\mathcal{L}\psi\|_{L^2(\Omega)}^2 +16 c_\Delta^2 L^2 \delta^{-3} (\max_j D_j )^2. 
\label{2-16}
\end{align}
Similarly to \eqref{2-3-0}--\eqref{2-6}, we can deduce that 
\begin{align}
\left|\int_\Omega  \phi_b' u_z \mathcal{L}\psi \dd \Bx\right| 
& \leq c_\Delta \delta ^{-1} \|u_z\|_{L^2(\Omega)}\|\mathcal{L} \psi\|_{L^2(\Omega)} \notag \\
& \leq 2c_\Delta (\|\partial_x u_x\|_{L^2(\Omega)}+\|\partial_y u_y\|_{L^2(\Omega)})\|\mathcal{L} \psi\|_{L^2(\Omega)} \notag \\
& \leq  \frac{4(\max_{j} K_j)^2 c_\Delta  }{\min_{j} K_j} (\|\partial_x \psi\|_{L^2(\Omega)}+ \|\partial_y \psi\|_{L^2(\Omega)})\|\mathcal{L} \psi\|_{L^2(\Omega)} \notag \\
& \leq \frac{1}{4}\|\mathcal{L} \psi\|_{L^2(\Omega)}^2+\frac{16(\max_{j} K_j)^4 c_\Delta^2  }{(\min_{j} K_j)^2}\|\nabla\psi\|_{L^2(\Omega)}^2. 
\label{2-17}
\end{align}
Furthermore, using H\"older's inequality, Young's inequality, \eqref{Sharps}$_1$  and Lemmas \ref{W-embedding}, \ref{lemmaA1}, we have 
\begin{align}
&   \left|\int_\Omega  \Bu\cdot\nabla \psi \mathcal{L}\psi \dd \Bx\right| \notag \\
& \quad \leq    \|\Bu\|_{L^r(\Omega)} \|\nabla \psi\|_{L^\frac{2r}{r-2}(\Omega)}  \|\mathcal{L} \psi\|_{L^2(\Omega)} \notag \\
& \quad \leq   (1+C_p)\max_j K_j \|\psi\|_{L^r(\Omega)} \|\nabla \psi\|_{L^\frac{2r}{r-2}(\Omega)}  \|\mathcal{L} \psi\|_{L^2(\Omega)} \notag \\
& \quad \leq  (1+C_p) C_2 C_u^\frac{3}{r}\max_j K_j \|\psi\|_{L^r(\Omega)} \|\nabla \psi\|_{L^2(\Omega)} ^{1- \frac{3}{r}}  \|\mathcal{L} \psi\|_{L^2(\Omega)}^{1+\frac3r} \notag  \\ 
& \quad \leq \frac{1}{4}\|\mathcal{L} \psi\|_{L^2(\Omega)}^2+ C_s\big[(1+C_p) C_2 C_u^\frac{3}{r}\max_j K_j\big]^\frac{2r}{r-3} \|\psi\|_{L^r(\Omega)}^\frac{2r}{r-3} \|\nabla \psi\|_{L^2(\Omega)}^2.
\label{2-18}
\end{align}
where $r =\frac{6}{3-2s} $ and $C_s>0$ is a uniform constant depending only on $s$. 
Substituting the estimates \eqref{2-16}--\eqref{2-18} into \eqref{2-15}, we obtain
\begin{equation}\label{2-19}
\frac{\dd}{\dd t} \|\sqrt{D}\nabla \psi\|_{L^2(\Omega)}^2  +   \|\mathcal{L}\psi\|_{L^2(\Omega)}^2 \leq M_2 + \left(M_3+M_4\|\psi\|_{L^r(\Omega)}^\frac{2r}{r-3}\right)  \|\sqrt{D}\nabla \psi\|_{L^2(\Omega)}^2 
\end{equation} 
for any $t>0$, with constants $M_2,M_3,M_4$ given in \eqref{M2M3} and \eqref{M4}. Set
\begin{equation}\nonumber
y(t) := \|\sqrt{D}\nabla \psi\|_{L^2(\Omega)}^2, \quad g(t) := M_3+M_4\|\psi\|_{L^r(\Omega)}^\frac{2r}{r-3},\quad h(t) := M_2.
\end{equation}
Furthermore, if $t\ge \max\{T_0,T_1\}$, it follows from Lemma \ref{L^r-estimate} and Lemma \ref{prop2} that
\begin{equation}\nonumber
  \int_{t}^{t+1} y(\xi)\dd \xi   \leq \frac{M_1H^2 }{\min_{j}D_j} +1+M_1
\end{equation} 
and 
\begin{equation}\nonumber
  \int_{t}^{t+1} g(\xi)\dd \xi \leq M_3 +M_4\sup_{\xi\in[t,t+1]} \|\psi(\xi)\|_{L^r(\Omega)}^\frac{2r}{r-3} \leq M_3+ M_4  \left( \frac{rM_0 H^2}{2(r-1)\min_{j}D_j}+1\right)^\frac{2}{r-3}.
\end{equation} 
Then an application of Lemma \ref{lemmaUGL} gives
\begin{equation}\nonumber
y(t+1)  \leq  \left(M_2+  \frac{M_1H^2}{\min_{j}D_j}+1+M_1\right) \operatorname{exp}\left[M_3+ M_4  \left( \frac{rM_0 H^2}{2(r-1)\min_{j}D_j}+1\right)^\frac{2}{r-3}\right],
\end{equation}
for all $t\ge  \max\{T_0,T_1\}$. This completes the proof.
\end{proof}

\begin{proof}[Proof of Theorem \ref{attractor-sharp}]
Let $B(0,R) \subset \mathcal{H}^s$ be a closed ball centered at $0$ with radius $R>0$. Assume the initial datum $\psi_0\in B(0,R)$. By the Sobolev embedding theorem, we find that 
\[ \|\psi_0\|_{L^2(\Omega)} +  \|\psi_0\|_{L^r(\Omega)} \leq C\|\psi_0\|_{\mathcal{H}^s} < CR =:R',\quad \text{ where }r =\frac{6}{3-2s}.\]
Define the set 
\begin{equation}\label{def:B1}
  B_1 := \left\{\psi\in \mathcal{V}:~\|\psi\|_{L^2(\Omega)}^2 \leq  \frac{M_1H^2}{\min_{j}D_j}+1,~~~ \| \nabla\psi\|_{L^2(\Omega)}^2 \leq  \frac{M_5}{\min_j D_j}  \right\},
\end{equation}
where $M_1$ is given by \eqref{M1-delta} and $M_5$ is determined in \eqref{2-14}. Thanks to Lemmas \ref{prop2} and \ref{prop3}, the solution $\psi(t)=S_0(t)\psi_0$ belongs to the set $B_1$ for any $t \ge \max \{T_0(R'), T_1(R')\}+1$. Since $B_1$ is bounded in $\mathcal{V}$ and relatively compact in $\mathcal{H}^s$, it serves a compact absorbing set for the semigroup $\{S_0(t)\}_{t\geq 0}$ in the phase space $\mathcal{H}^s$. Applying the abstract result \cite[Theorem 1.1]{Temam1997}, we can conclude the existence of the global attractor $\mathscr{A}_0$ for $\{S_0(t)\}_{t\geq 0}$, which is  given by 
\begin{equation}\nonumber
\mathscr{A}_0 = \bigcap_{\xi\geq 0} \overline{\bigcup_{t\geq \xi} S_0(t)B_1}.
\end{equation}
The proof is complete.
\end{proof}

\subsection{Diffuse interface model} 
We proceed to investigate the diffuse interface model \eqref{Diffuse-s}--\eqref{bc-epsilon}. The following two lemmas provide estimates similar to those in Lemmas \ref{prop2} and \ref{prop3}. The key point is that these estimates are independent of the parameter $\varepsilon$.

\begin{lemma}\label{prop4}
Let $\psi_0^\varepsilon\in \mathcal{H}$ and the constant $\delta$ in \eqref{phi_b1} satisfy \eqref{delta_2}. Then the solution $\psi$ to problem  \eqref{Diffuse-s}--\eqref{bc-epsilon} satisfies 
\begin{equation}\label{3-2}
\|\psi^\varepsilon(t)\|_{L^2(\Omega)}^2 \leq \|\psi_0^\varepsilon\|_{L^2(\Omega)}^2 e^{-\frac{\min_{j}D_j}{H^2}t} +\frac{M_1H^2}{\min_{j}D_j}\left(1- e^{-\frac{\min_{j}D_j}{H^2}t}\right),\quad \forall\, t\geq 0,
\end{equation}
where $M_1$ is given by \eqref{M1-delta}. Moreover, if $$t\ge T_1(\|\psi_0^\varepsilon\|_{L^2(\Omega)}) :=\frac{H^2}{\min_{j}D_j}\ln (\|\psi_0^\varepsilon\|_{L^2(\Omega)}^2+1),$$ 
then we have  
\begin{equation}\label{3-3}
\|\psi^\varepsilon(t) \|_{L^2(\Omega)}^2\leq  \frac{M_1H^2}{\min_{j}D_j}+1,\quad \int_t^{t+1} \|\sqrt{D^\varepsilon}\nabla\psi^\varepsilon(\tau) \|_{L^2(\Omega)}^2 \dd \tau \leq  \frac{M_1H^2}{\min_{j}D_j}+1+M_1.
\end{equation}
\end{lemma}

\begin{lemma}\label{prop6}
Let $s\in (\frac12,1)$ and $\psi_0^\varepsilon\in \mathcal{H}^s$. We take $r=\frac{6}{3-2s}$ and the constant $\delta$ in \eqref{phi_b1} satisfying  \eqref{def-delta}.   
For any $\tau\in (0,1)$, the solution $\psi^\varepsilon$ of problem \eqref{Diffuse-s}--\eqref{bc-epsilon} satisfies $\psi^\varepsilon\in L^\infty(\tau, \infty;\mathcal{V})$. Moreover, for any $t\ge \max\{ T_0,T_1\}+1$, where $T_0$, $T_1$ are given in Lemma \ref{L^r-estimate-diffuse}, Lemma \ref{prop4}, respectively, 
it holds 
\begin{equation} \label{4-14}
\begin{aligned}
\|\sqrt{D^\varepsilon} \nabla \psi(t)\|_{L^2(\Omega)}^2  
&\leq  M_5,
\end{aligned}
\end{equation}
with $M_5$ defined as in \eqref{2-14}.
\end{lemma}

\begin{remark}
The proof for Lemmas \ref{prop4}, \ref{prop6} 
follows the same argument as that for Lemmas \ref{prop2}, \ref{prop3}, taking advantage of the fact that $K^\varepsilon, D^\varepsilon$ have the same upper and lower bounds as $K, D$. In particular, in the derivation of the higher-order estimate \eqref{4-14}, we should use Lemma \ref{lemma-equivalence} instead of Lemma \ref{W-embedding}. The details are omitted. 
\end{remark}

Similarly to Theorem \ref{attractor-sharp}, for every $0<\varepsilon\ll 1$, we can apply Lemmas \ref{prop4}, \ref{prop6} to establish the existence of a global attractor for the semigroup $\{S_\varepsilon(t)\}_{t\geq 0}$ associated with the diffuse interface model \eqref{Diffuse-s}--\eqref{bc-epsilon}.
\begin{theorem}\label{attractor-Diffuse}
Let $s\in (\frac12,1)$. For any $0<\varepsilon\ll 1$, the semigroup $\{S_\varepsilon(t)\}_{t\geq  0}$ associated with the diffuse interface model \eqref{Diffuse-s}--\eqref{bc-epsilon} possesses a global attractor $\mathscr{A}_\varepsilon \subset \mathcal{H}^s$, which is bounded in $\mathcal{V}$, compact, and connected in $\mathcal{H}^s$. In particular, we have 
\begin{equation}\nonumber
\mathscr{A}_\varepsilon = \bigcap_{\xi\geq 0} \overline{\bigcup_{t\geq \xi} S_\varepsilon(t)B_1},
\end{equation}
where the set $B_1$ is given by \eqref{def:B1}.
\end{theorem}

\section{Asymptotics in the Long-time Regime}\label{sec-convergence}

In this section, we study the asymptotic behavior of global attractors $\mathscr{A}_\varepsilon$ and the stationary statistical properties of the diffuse interface model \eqref{Diffuse-s}--\eqref{bc-epsilon} as the width of the transition layer $\varepsilon\to 0$.

\subsection{Upper semi-continuity of the attractors}
First, we investigate the stability of the family of global attractors $\{\mathscr{A}_\varepsilon\}_{\varepsilon\geq 0}$ as $\varepsilon\to 0$.

\begin{theorem}\label{thm1}
Let $s\in (\frac12,1)$. 
The family of global attractors $\{\mathscr{A}_\varepsilon\}_{\varepsilon\geq 0}$ constructed in Theorems \ref{attractor-sharp}, \ref{attractor-Diffuse}
is upper-semicontinuous at $\varepsilon= 0$, that is, 
\begin{equation}\nonumber
\lim_{\varepsilon\to0} \mathrm{dist}(\mathscr{A}_{\varepsilon}, \mathscr{A}_0) =0,
\end{equation}
where $\mathrm{dist}(\mathscr{A}_{\varepsilon}, \mathscr{A}_0)$ is the Hausdorff semi-distance between $\mathscr{A}_{\varepsilon}, \mathscr{A}_0$ defined by 
\begin{equation}\nonumber
\mathrm{dist}(\mathscr{A}_{\varepsilon}, \mathscr{A}_0) := \sup_{f\in \mathscr{A}_{\varepsilon}} \inf_{g\in \mathscr{A}_0}\|f-g\|_{\mathcal{H}^s}.
\end{equation}
\end{theorem}

Since we are concerning the limit as $\varepsilon\to 0$, we can take a sufficiently small $\varepsilon_0>0$ and focus on the case $\varepsilon\in [0,\varepsilon_0]$ in the subsequent analysis. 

\begin{lemma}\label{prop3-4}
Let $s\in (\frac12, 1)$. For $\varepsilon\in [0,\varepsilon_0]$, we have 
$$
\cup_{\varepsilon\in [0,\varepsilon_0]}\mathscr{A}_{\varepsilon}\subset B_1,
$$
where the set $B_1$ is given by \eqref{def:B1}.
Suppose that the initial datum $\psi_0^\varepsilon \in B_1$, $\varepsilon\in [0,\varepsilon_0]$. Then the solution  $\psi^\varepsilon$ to the sharp interface model \eqref{Sharps}--\eqref{bc} (for $\varepsilon=0$) or the diffuse interface model \eqref{Diffuse-s}--\eqref{bc-epsilon} (for $\varepsilon\in (0,\varepsilon_0]$) satisfies
\begin{equation}\label{ep-uniform}
\|\psi^\varepsilon(t)\|_{L^2(\Omega)}^2 \leq \frac{M_1H^2}{\min_{j}D_j}+1,\quad \|\nabla \psi^\varepsilon(t)\|_{L^2(\Omega)}^2 \leq M_6, 
\quad \forall\, t\geq 0,
\end{equation} 
where the constant $M_6\geq 1$ depends on $D_j, K_j, M_0, M_1, M_5, s, \Omega$. Moreover, it holds 
\begin{equation}\label{ep-uniform-b}
\int_0^t\|\mathcal{L}^\varepsilon \psi^\varepsilon\|_{L^2(\Omega)}^2\dd s \leq M_7(t),
\quad \forall\, t\geq 0, 
\end{equation}
with $M_7(t)>0$ determined by \eqref{3-34}.
\end{lemma}
\begin{proof}
We recall that $B_1\subset \mathcal{V}$ (see \eqref{def:B1}) serves as a compact absorbing set for $\{S_\varepsilon(t)\}_{t\geq 0}$, $\varepsilon\in[0,\varepsilon_0]$. By the definition of the global attractor, we have $\mathscr{A}_{\varepsilon}\subset B_1$ for every $\varepsilon\in[0,\varepsilon_0]$, so that the family of global attractors $\{\mathscr{A}_{\varepsilon}\}_{\varepsilon\in [0,\varepsilon_0]}$ is uniformly bounded in $H^1(\Omega)$.


Concerning the estimate \eqref{ep-uniform}, we only deal with the case $\varepsilon=0$, the case $\varepsilon>0$ can be treated in a similar manner with the same bounds. First, by \eqref{2-2}, we have 
\begin{equation}\nonumber
\begin{aligned}
\|\psi(t)\|_{L^2(\Omega)}^2 & \leq   \left(\frac{M_1H^2}{\min_{j}D_j}+1\right) e^{-\frac{\min_{j}D_j}{H^2}t} +\frac{M_1H^2}{\min_{j}D_j}\left(1- e^{-\frac{\min_{j}D_j}{H^2}t}\right)\\
\leq &~\frac{M_1H^2}{\min_{j}D_j}+1,\quad \forall\,t\ge 0.
\end{aligned}
\end{equation}
Integrating \eqref{2-8} yields 
\begin{align}
\|\psi(t)\|_{L^2(\Omega)}^2+ \int_0^{t} \|\sqrt{D}\nabla\psi(\tau) \|_{L^2(\Omega)}^2 \dd \tau 
& \leq \|\psi_0\|_{L^2(\Omega)}^2 + M_1 t \notag \\
& \leq \frac{M_1H^2}{\min_{j}D_j}+1 + M_1 t,\quad \forall\,t\ge 0.
\label{2-31}
\end{align}
By interpolation, for $r=\frac{6}{3-2s}$, we also have  
\begin{equation}\nonumber
\|\psi_0\|_{L^r(\Omega)} \leq C_s\|\psi_0\|_{L^2(\Omega)}^{1-s}\|\sqrt{D}\nabla \psi_0\|_{L^2(\Omega)}^s  \leq C_s\left(\frac{M_1H^2}{\min_{j}D_j}+1\right)^\frac{1-s}{2} M_5^\frac{s}{2}.
\end{equation}
This, together with Lemma \ref{L^r-estimate}, gives 
\begin{equation}\label{2-32}
\begin{aligned}
\|\psi(t)\|_{L^r(\Omega)}^r
& \leq C_s\left(\frac{M_1H^2}{\min_{j}D_j}+1\right)^\frac{3(1-s)}{3-2s} M_5^\frac{3s}{ 3-2s}  +\frac{r M_0 H^2}{2(r-1)\min_{j}D_j}=: \widetilde{M}, \quad \forall\, t\geq 0.
\end{aligned}
\end{equation}
With the aid of \eqref{2-32}, the classical Gronwall lemma on a finite interval $[0,\max\{T_0,T_1\}+1]$ and the uniform Gronwall lemma on $[\max\{T_0,T_1\}+1,\infty)$, we can deduce from \eqref{2-19} that 
\begin{equation}\label{3-33}
\begin{aligned}
\|\sqrt{D}\nabla \psi(t)\|_{L^2(\Omega)}^2 \leq M_5',\quad \forall\, t\geq 0,
\end{aligned}
\end{equation}
for some $M_5'$ only depending on $D_j, K_j, M_0, M_1, M_5, s$. Then we set 
$$
M_6=\frac{M_5'}{\min_j D_j},
$$
which yields the estimate \eqref{ep-uniform}. 
Moreover, integrating \eqref{2-19} with respect to time, we infer from \eqref{2-31}, \eqref{2-32}  that 
\begin{align}
& \frac12\int_0^t \|\mathcal{L}\psi(\tau) \|_{L^2(\Omega)}^2 \dd \tau  \notag \\  
  &\quad \leq   \|\sqrt{D}\nabla \psi_0\|_{L^2(\Omega)}^2 +\int_0^t M_2 + \left(M_3+M_4 \|\psi(\tau)\|_{L^r(\Omega)}^\frac{2r}{r-3}\right)  \|\sqrt{D}\nabla \psi(\tau)\|_{L^2(\Omega)}^2\dd \tau \notag \\
  & \quad \leq M_5  +  M_2 t+\left( M_3  + M_4\widetilde{M}^\frac{2}{r-3}\right) \left(\frac{M_1H^2}{\min_{j}D_j}+1+M_1 t \right) \notag \\
  &\quad =:\frac12 M_7(t). 
  \label{3-34}
\end{align}
The proof is complete.
\end{proof}

Lemma \ref{prop3-4} implies that $\cup_{\varepsilon\in [0,\varepsilon_0]}\mathscr{A}_{\varepsilon}$ is relatively compact in $\mathcal{H}^s$. Next, we prove the continuous dependence of solutions on the parameter $\varepsilon$. Let $(\Bu,p,\psi)$ and $(\Bu^\varepsilon,p^\varepsilon,\psi^\varepsilon)$ be solutions to the sharp interface model \eqref{Sharps}--\eqref{bc} and diffuse interface model \eqref{Diffuse-s}--\eqref{bc-epsilon}, respectively. Denote the difference of the solutions 
\[\widetilde{\Bu} = \Bu-\Bu^\varepsilon,\quad \widetilde{p} = p-p^\varepsilon,\quad 
\widetilde{\psi} = \psi-\psi^\varepsilon,\]
and the coefficients 
\[\widetilde{K} := K-K^\varepsilon,\quad 
\widetilde{D} := D - D^\varepsilon.\]
Then $(\widetilde{\Bu}, \widetilde{p},\widetilde{\psi})$ satisfies the following system
\begin{equation}\label{error} 
      \left\{\begin{aligned}
        &\widetilde{\Bu}=- K ^\varepsilon\big(\nabla \widetilde{p } + \widetilde{\psi} \Be_{z}\big)- \widetilde{K} \big(\nabla p  + \psi \Be_{z}\big),\\
        &\nabla \cdot \widetilde{\Bu} =0,\\
        & \partial_t \widetilde{\psi} + \Bu^\varepsilon\cdot\nabla \widetilde{\psi}+\widetilde{\Bu}\cdot\nabla \psi  +  \phi_b'\widetilde{\Bu} \cdot \Be_z   - \nabla \cdot (D^\varepsilon\nabla\widetilde{\psi})- \nabla \cdot (\widetilde{D} \nabla \psi)=0.
      \end{aligned}\right.  
  \end{equation}  
supplemented with the initial condition 
 \begin{equation} 
   \widetilde{\psi}|_{t=0}=0,
\end{equation}
as well as the boundary conditions 
 \begin{equation} 
\widetilde{\Bu}\cdot \Be_z|_{z=0,-H}= 0,\quad  \widetilde{\psi}|_{z=0}=0, \quad \widetilde{\psi}|_{z=-H}=0,
 \end{equation} 
with periodicity in the horizontal directions $(x,y)$. 
\begin{lemma}\label{prop-error}
   Suppose that $\psi_0 = \psi_0^\varepsilon \in B_1$, where $B_1$ is defined as in \eqref{def:B1}. Then $\widetilde{\psi}$ satisfies the following estimate
  \begin{equation}\label{diff-conti}
\|\widetilde{\psi}(t)\|_{L^2(\Omega)}^2 \leq M_8(t)\varepsilon^\frac16,\quad \forall\, t\geq 0,
  \end{equation}
  where $M_8(t)>0$ depends on $K_j$ ,$D_j$, $M_0$, $M_1$, $M_5$, $\Omega$ and $t$, see \eqref{5-11}.  
\end{lemma}
\begin{proof}
We observe that the differences of the coefficients $\widetilde{K}, \widetilde{D}$ are piecewise linear functions supported in a narrow layer of width $2\varepsilon$. Thus, we have 
\begin{equation}\label{5-1}
\|\widetilde{K}\|_{L^r(\Omega)} \leq C  \varepsilon^{\frac{1}{r}}, \quad \|\widetilde{D}\|_{L^r(\Omega)} \leq C \varepsilon^{\frac{1}{r}},\quad \forall\, r\in [2,\infty).
\end{equation}
Since $\widetilde{p}$ is determined by the elliptic problem
  \begin{equation}\label{error-pressure}
  \left\{\begin{aligned}
&-\nabla \cdot (K^\varepsilon\nabla {\widetilde{p}}) = \nabla \cdot  \big[K^\varepsilon\widetilde{\psi}\Be_{z} +\widetilde{K}(\nabla p+\psi \Be_z)\big], \quad \text{ in } \Omega,\\
&\partial_z \widetilde{p}(x,-H)  = \partial_z \widetilde{p}(x,0)  = 0,
  \end{aligned}\right. 
\end{equation}
with periodicity in the horizontal directions $(x,y)$, it follows from Lemma \ref{estimate-p} and \eqref{5-1} that 
\begin{equation}\nonumber
\begin{aligned}
\|\nabla \widetilde{p}\|_{L^4(\Omega)} 
& \leq  C (\|K^\varepsilon\widetilde{\psi}\Be_{z} \|_{L^4(\Omega)} +\|\widetilde{K}\nabla p \|_{L^4(\Omega)} +\|\widetilde{K} \psi\|_{L^4(\Omega)})\\
& \leq C\big(\|K^\varepsilon\|_{L^\infty(\Omega)}\|  \widetilde{\psi}\|_{L^4(\Omega)} +\|\widetilde{K}\|_{L^{12}(\Omega)}(\|\nabla p\|_{L^6(\Omega)} +\|\psi\|_{L^6(\Omega)})\big)\\
& \leq C \|\widetilde{\psi}\|_{L^4(\Omega)} 
+ C \varepsilon^\frac{1}{12} \|\psi\|_{L^6(\Omega)}.
\end{aligned}
\end{equation}
Thanks to the Sobolev embedding theorem and Lemma \ref{prop3-4}, we get
\begin{equation}\label{5-5a}
\|\nabla \widetilde{p}\|_{L^4(\Omega)} 
\leq C \|\widetilde{\psi}\|_{L^4(\Omega)}+ C \varepsilon^\frac{1}{12} \|\nabla\psi\|_{L^2(\Omega)} 
\leq C \|\widetilde{\psi}\|_{L^4(\Omega)} + CM_6^\frac12 \varepsilon^\frac{1}{12}.
\end{equation}
Consequently, it holds 
\begin{equation}\label{5-5}
\begin{aligned}
\|\widetilde{\Bu}\|_{L^4(\Omega)} 
& \leq \|K ^\varepsilon \|_{L^\infty(\Omega)}(\|\nabla \widetilde{p }\|_{L^4(\Omega)} 
+ \|\widetilde{\psi}\|_{L^4(\Omega)} )  
+ \|\widetilde{K}\|_{L^{12}(\Omega)} (\|\nabla p\|_{L^6(\Omega)}  + \|\psi\|_{L^6(\Omega)})\\
& \leq C \big(\| \widetilde{\psi}\|_{L^4(\Omega)} 
+ M_6^\frac12  \varepsilon^\frac{1}{12}  \big)  
+ C \varepsilon^\frac{1}{12} \|\psi\|_{L^{6}(\Omega)}\\
& \leq C \big(\| \widetilde{\psi}\|_{L^4(\Omega)} 
+ M_6^\frac12  \varepsilon^\frac{1}{12}  \big)  
+ C \varepsilon^\frac{1}{12} \|\nabla\psi\|_{L^2(\Omega)}\\
& \leq C \big(\| \widetilde{\psi}\|_{L^4(\Omega)} 
+ M_6^\frac12  \varepsilon^\frac{1}{12} \big).
\end{aligned}
\end{equation}
In a similar way, we can get
\begin{equation}\label{5-6}
\begin{aligned}
\|\nabla \widetilde{p}\|_{L^2(\Omega)} + \|\widetilde{\Bu}\|_{L^2(\Omega)} 
\leq C \|\widetilde{\psi}\|_{L^2(\Omega)}+ CM_6^\frac12 \varepsilon^\frac{1}{3}.
\end{aligned}
\end{equation}
Next, testing equation \eqref{error}$_3$ by $\widetilde{\psi}$ and using integration by parts, we obtain 
\begin{equation}\label{5-7}
\frac{1}{2}\frac{\dd }{\dd t} \|\widetilde{\psi}\|_{L^2(\Omega)}^2 + \|\sqrt{D^\varepsilon} \nabla \widetilde{\psi}\|_{L^2(\Omega)}^2 
+\int_\Omega \big[ \widetilde{\Bu}\cdot\nabla \psi  +  \phi_b'\widetilde{\Bu} \cdot \Be_z    
- \nabla \cdot (\widetilde{D} \nabla \psi)\big]\widetilde{\psi} \dd \Bx=0.
\end{equation}
With the aid of \eqref{5-5a}--\eqref{5-6}, Lemmas \ref{W-embedding}--\ref{lemmaA1} and Young's inequality, we can deduce that 
\begin{align}
\left|\int_\Omega  ( \widetilde{\Bu}\cdot\nabla \psi) \widetilde{\psi}\dd \Bx\right| 
& = \left|\int_\Omega (\widetilde{\Bu}\cdot\nabla  \widetilde{\psi} )\psi \dd \Bx\right| 
\leq  \|\nabla  \widetilde{\psi}\|_{L^2(\Omega)}\|\psi\|_{L^4(\Omega)}  \|\widetilde{\Bu}\|_{L^4(\Omega)} \notag \\ 
& \leq \frac{\min_j D_j}{8} \|\nabla  \widetilde{\psi}\|_{L^2(\Omega)}^2 
+  C \|\psi\|_{L^4(\Omega)}^2  \|\widetilde{\Bu}\|_{L^4(\Omega)}^2 \notag \\
& \leq \frac{\min_j D_j}{8} \|\nabla  \widetilde{\psi}\|_{L^2(\Omega)}^2 
+  C \|\nabla\psi\|_{L^2(\Omega)}^2 (\| \widetilde{\psi}\|_{L^4(\Omega)}^2 + M_6  \varepsilon^\frac16 ) \notag \\
& \leq \frac{\min_j D_j}{8} \|\nabla  \widetilde{\psi}\|_{L^2(\Omega)}^2 
+ C M_6 (\| \widetilde{\psi}\|_{L^4(\Omega)}^2 
+ M_6  \varepsilon^\frac16 ) \notag \\
& \leq \frac{\min_j D_j}{8} \|\nabla  \widetilde{\psi}\|_{L^2(\Omega)}^2 
+ C M_6 \| \widetilde{\psi}\|_{L^2(\Omega)} ^\frac12\| \nabla\widetilde{\psi}\|_{L^2(\Omega)}^\frac32 
+ CM_6^2  \varepsilon^\frac16 \notag  \\
& \leq \frac{\min_j D_j}{4} \|\nabla  \widetilde{\psi}\|_{L^2(\Omega)}^2 
+ C M_6^4 \| \widetilde{\psi}\|_{L^2(\Omega)}^2 
+ CM_6^2  \varepsilon^\frac16, 
\label{5-8}
\end{align}
\begin{align}
\left|\int_\Omega  \phi_b'\widetilde{\Bu} \cdot \Be_z \widetilde{\psi} \dd \Bx\right| 
& \leq c_\Delta\delta^{-1} \|\widetilde{\Bu}\|_{L^2(\Omega)}\|\widetilde{\psi}\|_{L^2(\Omega)}\notag \\
& \leq C(\| \widetilde{\psi}\|_{L^2(\Omega)} 
+ M_6^\frac12  \varepsilon^\frac13)\| \widetilde{\psi}\|_{L^2(\Omega)} \notag \\
& \leq C \| \widetilde{\psi}\|_{L^2(\Omega)}^2 
+ CM_6  \varepsilon^\frac23 
\label{5-9}
\end{align}
and
\begin{align}
\left|\int_\Omega - \nabla \cdot (\widetilde{D} \nabla \psi )\widetilde{\psi} \dd \Bx \right|
& = \left|\int_\Omega  \widetilde{D} \nabla \psi \cdot \nabla\widetilde{\psi}\dd \Bx \right| \notag \\
& \leq \frac{\min_jD_j}{4} \|\nabla  \widetilde{\psi}\|_{L^2(\Omega)}^2 
+ C \|\widetilde{D}\|_{L^3(\Omega)}^2 \|\nabla\psi\|_{L^6(\Omega)}^2 \notag \\
& \leq \frac{\min_jD_j}{4} \|\nabla  \widetilde{\psi}\|_{L^2(\Omega)}^2 
+ C \varepsilon^\frac23 \|\psi\|_{\mathcal{W}}^2 \notag \\
& \leq \frac{\min_jD_j}{4} \|\nabla  \widetilde{\psi}\|_{L^2(\Omega)}^2 
+ C \varepsilon^\frac23 \|\L\psi\|_{L^2(\Omega)}^2.
\label{5-10}
\end{align}
Substituting \eqref{5-8}--\eqref{5-10} into \eqref{5-7} yields  
\begin{equation} 
\frac{\dd }{\dd t} \|\widetilde{\psi}\|_{L^2(\Omega)}^2 
+ \min_j D_j \|\nabla \widetilde{\psi}\|_{L^2(\Omega)}^2 
\leq CM_6^4 \|\widetilde{\psi}\|_{L^2(\Omega)}^2 
+ CM_6^2\varepsilon^\frac16 
+ C\varepsilon^\frac23\|\L\psi\|_{L^2(\Omega)}^2,
\end{equation}
where we have used the facts $M_6\geq 1$ and $\varepsilon\in (0,1)$. 
Since $\widetilde{\psi}(0)=0$, we infer from Gronwall's lemma and Lemma \ref{prop3-4} that 
\begin{align}
\|\widetilde{\psi}(t)\|_{L^2(\Omega)}^2 
& \leq \int_0^t \big(CM_6^2\varepsilon^\frac16+ C\varepsilon^\frac23\|\L\psi(\tau )\|_{L^2(\Omega)}^2  \big) e^{\int_\tau ^t CM_6^4\dd \xi} \dd \tau \notag \\
& \leq C e^{CM_6^4t} \left( M_6^2  \varepsilon^\frac16 t 
+ \varepsilon^\frac23 M_7(t)\right) \notag \\
& \leq C e^{CM_6^4 t} \left( M_6^2 t  + M_7(t)   \right)\varepsilon^\frac16 =:M_8(t) \varepsilon^\frac16, \quad \forall\, t\geq 0. \label{5-11}
\end{align}
The proof is complete. 
\end{proof}

Now we are in a position to prove Theorem \ref{thm1}.  

\begin{proof}[Proof of Theorem \ref{thm1}] 
We assume by contradiction that the assertion is false. Then there exists a constant $c >0$ and a decreasing sequence $\{\varepsilon_j\}_{j=1}^\infty\subset (0,\varepsilon_0]$ converging to $0$ such that
\[ \mathrm{dist}(\mathscr{A}_{\varepsilon_j},\mathscr{A}_0) \geq c ,\quad \forall\, j \geq 1.
\]
Since $\mathscr{A}_{\varepsilon_j}$ is compact in $\mathcal{H}^s$, there exists an $a_j \in \mathscr{A}_{\varepsilon_j}$ such that
\[\dist (a_j,\mathscr{A}_0) =\dist (\mathscr{A}_{\varepsilon_j},\mathscr{A}_0) \geq c.\]
On the other hand, since $\mathscr{A}_0$ attracts all the bounded sets in $\mathcal{H}^s$, we can choose $T_*=T_*(B_1) > 0$ sufficiently large such that
\[\dist (S_0(T_*)B_1 ,\mathscr{A}_0)<\frac{c}{4}.\]
The global attractor $\mathscr{A}_\varepsilon$ is invariant with respect to $S_\varepsilon(t)$ for any $t\ge 0$, i.e., $S_\varepsilon (t) \mathscr{A}_\varepsilon= \mathscr{A}_\varepsilon$. Hence, for each $a_j \in \mathscr{A}_{\varepsilon_j}$, there exists $b_j \in \mathscr{A}_{\varepsilon_j}$ such that $S_{\varepsilon_j}(T_*)b_j= a_j$ for all $j\in \mathbb{Z}^+$. This implies that $b_j\in \mathscr{A}_{\varepsilon_j} \subset B_1$. By Lemmas \ref{prop3-4} and \ref{prop-error}, we have 
\begin{align}
 \|a_j -S_0(T_*)b_j\|_{\mathcal{H}^s} 
& \leq C_s\|a_j -S_0(T_*)b_j\|_{L^2(\Omega)}^{1-s}  \|\nabla(a_j -S_0(T_*)b_j )\|_{L^2(\Omega)}^s \notag \\
& \leq C_s\big[M_8(T_*)\varepsilon_j^\frac16\big]^\frac{1-s}{2}\big(\|\nabla (S_{\varepsilon_j}(T_*)b_j)\|_{L^2}+\|\nabla (S_0(T_*)b_j)\|_{L^2(\Omega)}\big)^{s} \notag \\
& \leq 2C_s M_6^\frac{s}{2}\big[M_8(T_*)\big]^\frac{1-s}{2} \varepsilon_j^\frac{1-s}{12} <\frac{c}{4}, \notag 
\end{align}
provided that $j$ is sufficiently large. Hence, it holds 
\begin{align*} 
  \dist(a_j, \mathscr{A}_0) 
  & \leq  \|a_j -S_0(T_*)b_j\|_{\mathcal{H}^s} 
  + \dist(S_0(T_*)b_j, \mathscr{A}_0) \\
  & \leq  \|a_j -S_0(T_*)b_j\|_{\mathcal{H}^s} + \dist(S_0(T_*)B_1, \mathscr{A}_0) 
  \leq \frac{c}{2},
\end{align*}
which leads to a contradiction. The proof is complete. 
\end{proof}

\subsection{Upper semi-continuity of invariant measures}
Now we study the stationary statistical properties of the diffuse interface model \eqref{Diffuse-s}--\eqref{bc-epsilon} in terms of invariant measures (see \cite{Wang2008b}). To this end, we recall the following definition:
\begin{definition} Consider an abstract continuous dynamical system $\{S(t)\}_{t\geq 0}$ on a Hilbert space $X$. A Borel measure $\mu$ is called an invariant measure of $\{S(t)\}_{t\geq 0}$, if
\[
\mu(E)=\mu(S^{-1}(t)E), 
\]
for any $t\ge 0$ and Borel measurable set $E$. 
\end{definition}
Since invariant measures are supported in the global attractors \cite{Wang2009}, the upper semi-continuity of the global attractors established in Theorem \ref{thm1} indicates that the statistical properties of the diffuse interface model \eqref{Diffuse-s}--\eqref{bc-epsilon} may be close to those of the sharp interface model \eqref{Sharps}--\eqref{bc}. 

For any $0<\varepsilon \ll 1$, we denote by $\mathcal{IM}_\varepsilon$ the set of all invariant measures of the dynamical system $(S_\varepsilon(t), \mathcal{H}^{s})$ corresponding to the sharp interface model ($\varepsilon=0$) and the diffuse interface model ($\varepsilon>0$). 
%
%
Lemma \ref{prop3-4} implies that the semigroups $\{S_\varepsilon(t)\}_{t\geq 0}$, $\varepsilon\in (0,\varepsilon_0]$ are uniformly dissipative. On the other hand, by Lemma \ref{prop-error}, the uniform estimate \eqref{ep-uniform} and the interpolation inequality, we find 
\begin{equation}\notag 
\|\widetilde{\psi}(t)\|_{\mathcal{H}^s} \leq C M_6^\frac{s}{2}\big[M_8(t)\big]^\frac{1-s}{2} \varepsilon^\frac{1-s}{12},\quad \forall\, t\geq 0.
  \end{equation}
This implies that the trajectories of $S_\varepsilon$ converge to those of $S_0$ in any time interval $[0,T]$ uniformly on the attractors $\mathscr{A}_\varepsilon$. Hence, applying the abstract result \cite[Theorem 4]{Wang2009}, we can conclude 
\begin{theorem}\label{thm2}
Let $s\in (\frac12,1)$. The sets of invariant measures $\{\mathcal{IM}_\varepsilon\}_{\varepsilon\in [0,\varepsilon_0]}$ are upper semi-continuous in the sense that for any sequence $\{\mu_\varepsilon \in \mathcal{IM}_\varepsilon,\,\varepsilon\in (0,\varepsilon_0]\}$, there exist a $\mu_0\in \mathcal{IM}_{0}$ and a subsequence (still denoted by $\{\mu_{\varepsilon}\}$) such that $\mu_\varepsilon$ converges to $\mu_0$ as $\varepsilon\to 0$ in the weak sense.
\end{theorem}

\subsection{Upper semi-continuity of the Nusselt number} 
The Nusselt number defined via long time averages measures the concentration transport in the vertical direction (see \cite{doering1998jfm,Wang2008}):
\begin{equation}
\mathrm{Nu}_\varepsilon = \sup_{\psi^\varepsilon_0 \in \mathcal{H}}\limsup_{T\to\infty} \frac{1}{T} \int_0^T  \frac{1}{|\Omega|}\int_{\Omega} \Bu^\varepsilon\cdot \Be_z \psi^\varepsilon \dd \Bx \dd t.
\end{equation}
%
As a corollary of Theorem \ref{thm2}, we can obtain the upper semi-continuity of the Nusselt number as $\varepsilon \to 0$. See, e.g., \cite[Theorem 5]{Wang2009} for details of the proof.  
\begin{theorem}\label{thm3}
  There exist ergodic invariant measures $\nu_\varepsilon \in \mathcal{IM}_\varepsilon$ and $\nu_0 \in \mathcal{IM}_{0}$ such that
\[
\mathrm{Nu}_\varepsilon = \int_{\mathcal{H}} \int_\Omega \Bu_0^\varepsilon\cdot\Be_z\psi_0^\varepsilon\dd \Bx\dd \nu_\varepsilon(\psi_0^\varepsilon)
\quad 
\text{and}
\quad 
\mathrm{Nu}_0 = \int_{\mathcal{H}} \int_\Omega   \Bu_0\cdot\Be_z\psi_0\dd \Bx\dd \nu_0(\psi_0).\]
Moreover, the Nusselt number is upper semi-continuous
with respect to the parameter $\varepsilon$ in the sense that
\[
\limsup_{\varepsilon\to 0}\mathrm{Nu}_\varepsilon \leq  \mathrm{Nu}_0.
\]
\end{theorem}

\section{Finite Dimensionality of Global Attractors}
Attractors are usually fractal subsets of the phase space. Finite dimensionality is an important property characterizing the structure of global attractors. Let us recall the definition of fractal dimension (cf. \cite[Definition 4.1]{Ze00}): 

\begin{definition}
Let $\mathcal{M}$ be a precompact set in a metric space $\mathcal{X}$. $N_\eta(\mathcal{M},\mathcal{X})$ is the minimal number of closed balls in $\mathcal{X}$ of radius $\eta$ that cover the set $\mathcal{M}$. Define the Kolmogorov's entropy of $\mathcal{M}$ in $\mathcal{X}$ as 
$$
\mathcal{H}_\eta(\mathcal{M},\mathcal{X})=\log_2 N_\eta(\mathcal{M},\mathcal{X}).
$$
Then the fractal (box-counting) dimension of $\mathcal{M}$ in $\mathcal{X}$ is defined by
$$
\mathrm{dim}_\mathrm{f} (\mathcal{M}, \mathcal{X}) = \limsup_{\eta\to 0} \frac{\mathcal{H}_\eta(\mathcal{M},\mathcal{X})}{\log_2 (1/\eta)},
$$
\end{definition}

The following result due to \cite[Theorem 4.1]{Ze00} gives a general method to prove the finite
fractal dimensionality of a compact set.
\begin{theorem}\label{zelik-thm}
Let $E_1$ and $E$ be Banach spaces, $E_1$ be compactly embedded in $E$ and let $X\subset E_1$. Assume that there exists a map $\mathcal{S}:X\to X$, such that $\mathcal{S}(X)=X$, and the following `smoothing' property is valid
\begin{equation}
\|\mathcal{S}a_1-\mathcal{S}a_2\|_{E_1}\leq C \|a_1-a_2\|_E, \quad \forall\, a_1,a_2\in X.
\label{smooth2}
\end{equation}
Then the fractal dimension of $X$ in $E$ is finite and can be estimated in the following way 
$$
\mathrm{dim}_\mathrm{f}(X,E)\leq \mathcal{H}_{\frac{1}{4C}}(B(1,0,E_1),E),
$$
where $C$ is the same as in \eqref{smooth2} and $B(1,0,E_1)$ means the unit ball in $E_1$.
\end{theorem}

The main result of this section is the following
\begin{theorem}\label{finite-dim}
Let $s\in (\frac12,1)$. For every $\varepsilon\in [0,\varepsilon_0]$, the global attractor $\mathscr{A}_\varepsilon$ constructed in Theorems \ref{attractor-sharp}, \ref{attractor-Diffuse} has a finite fractal dimension. Moreover, we have
\begin{align}
\mathrm{dim}_\mathrm{f}(\mathscr{A}_\varepsilon,\mathcal{H}^s)\leq \mathcal{H}_{\frac{1}{4C_A}}(B(1,0,\mathcal{V}),\mathcal{H}^s),\quad \forall\, \varepsilon\in [0,\varepsilon_0],
\label{A-dim}
\end{align}
where the constant $C_A>0$ depends only on $D_j$, $K_j$, $\Omega$, $c_\Delta$, $\delta$, $M_1$ and $M_5$, but is independent of the parameter $\varepsilon$.
\end{theorem}

To prove Theorem \ref{finite-dim}, it remains to verify the smoothing property for the semigroups $\{S_\varepsilon(t)\}_{t\geq 0}$, $\varepsilon\in [0,\varepsilon_0]$. We first consider the case $\varepsilon=0$:

\begin{lemma}\label{lem-smooth}
Let $s\in (\frac12,1)$. Suppose $(\Bu_1, p_1, \psi_1)$ and $(\Bu_2, p_2, \psi_2)$ are two solutions to the sharp interface model \eqref{Sharps}--\eqref{bc} corresponding to the initial data $\psi_{1,0}, \psi_{2,0}\in \mathscr{A}_0$, respectively. Then we have 
\begin{equation}\label{smooth}
\|\nabla \psi_1 (t)- \nabla \psi_2(t)  \|^2_{L^2(\Omega)}
\leq M_{11}\left(\frac{1}{t}+t^2\right)e^{M_9 t} \|\L^{\frac{s}{2}}(\psi_{1,0}-\psi_{2,0})\|_{L^2(\Omega)}^2, \quad \forall\, t>0,
\end{equation}
where $M_9$, $M_{11}$ are positive constants  depending only on $D_j$, $K_j$, $\Omega$, $c_\Delta$, $\delta$, $M_1$ and $M_5$. 
\end{lemma}
\begin{proof}
The difference of solutions 
$(\overline{\Bu}, \overline{\psi}, \overline{p}) : = (\Bu_1 - \Bu_2, p_1 - p_2, \psi_1 - \psi_2)$ satisfies the system \eqref{H^s-6} with the initial condition \eqref{H^s-6-ini}. Integrating \eqref{H^s-8} with respect to time and using \eqref{S-conti-Hs}, we get
\begin{align}
& \int_0^t \|\sqrt{D} \nabla \overline{\psi}(\tau)\|^2_{L^2(\Omega)}\,\mathrm{d}\tau \leq \int_0^t \|\mathcal{L}^s\overline{\psi}(\tau)\|^2_{L^2(\Omega)}\,\mathrm{d}\tau 
\notag \\
& \quad \leq \|\L^{\frac{s}{2}}(\psi_{1,0}-\psi_{2,0})\|_{L^2(\Omega)}^2 + C
\int_0^t \|\mathcal{L}^\frac{s}{2}\overline{\psi}(\tau)\|_{L^2(\Omega)}^2 \left(\|\psi_1(\tau)\|_{\mathcal{H}^s}^\frac{4}{2s-1}+ \| \psi_2(\tau)\|_{\mathcal{H}^s}^\frac{4}{2s-1}+1\right)\,\mathrm{d}\tau 
\notag \\
&\quad \leq  \|\L^{\frac{s}{2}}(\psi_{1,0}-\psi_{2,0})\|_{L^2(\Omega)}^2 \left(1+M_{9}t e^{M_9t}\right).
\label{dim-0}
\end{align}
Here, we use the fact that $\psi_1,\psi_2\in \mathscr{A}_0$ are uniformly bounded in $\mathcal{V}$ and take
$$
M_9=C\big(M_5^\frac{2}{2s-1}+1\big).    
$$
Next, testing equation \eqref{H^s-6}$_3$ by $\mathcal{L}\overline{\psi}$ and using integration by parts, we obtain
\begin{equation}\label{dim-1}
\frac{1}{2}\frac{\dd}{\dd t}\|\sqrt{D}\nabla \overline{\psi}\|_{L^2(\Omega)}^2 + \|\mathcal{L}\overline{\psi}\|_{L^2(\Omega)}^2 
= -(\Bu_1\cdot\nabla \overline{\psi}, \L\overline{\psi})-(\overline{\Bu}\cdot\nabla \psi_2, \L\overline{\psi}) - (\phi_b'\overline{\Bu} \cdot \Be_z, \L\overline{\psi}).
\end{equation}
The right-hand side of \eqref{dim-1} can be estimated as follows: 
\begin{align}
\left|\int_\Omega  \phi_b' \overline{u}_z \mathcal{L}\overline{\psi} \dd \Bx\right| 
& \leq c_\Delta \delta ^{-1} \|\overline{u}_z\|_{L^2(\Omega)}\|\mathcal{L} \overline{\psi}\|_{L^2(\Omega)} \notag \\
& \leq 2c_\Delta (\|\partial_x \overline{u}_x\|_{L^2(\Omega)}+\|\partial_y \overline{u}_y\|_{L^2(\Omega)})\|\mathcal{L} \overline{\psi}\|_{L^2(\Omega)}  \notag \\
& \leq  \frac{4(\max_{j} K_j)^2 c_\Delta  }{\min_{j} K_j} (\|\partial_x \overline{\psi}\|_{L^2(\Omega)} 
+ \|\partial_y \overline{\psi}\|_{L^2(\Omega)})\|\mathcal{L} \overline{\psi}\|_{L^2(\Omega)}  \notag \\
& \leq \frac{1}{4}\|\mathcal{L} \overline{\psi}\|_{L^2(\Omega)}^2+\frac{16(\max_{j} K_j)^4 c_\Delta^2  }{(\min_{j} K_j)^2}\|\nabla\overline{\psi}\|_{L^2(\Omega)}^2.
\label{dim-2}
\end{align}
Furthermore, using H\"older's inequality, Young's inequality, \eqref{Sharps}$_1$  and Lemmas \ref{W-embedding}, \ref{lemmaA1}, we have 
\begin{align}
&  \left|\int_\Omega  \Bu_1\cdot\nabla \overline{\psi} \mathcal{L}\overline{\psi} \dd \Bx\right| \notag \\
& \quad \leq  \|\Bu_1\|_{L^6(\Omega)} \|\nabla \overline{\psi}\|_{L^3(\Omega)}  \|\mathcal{L} \overline{\psi}\|_{L^2(\Omega)}  \notag \\
& \quad \leq   (1+C_p)\max_j K_j \|\psi_1\|_{L^6(\Omega)} \|\nabla \overline{\psi}\|_{L^3(\Omega)}  \|\mathcal{L} \overline{\psi}\|_{L^2(\Omega)}  \notag \\
& \quad \leq  (1+C_p) C_2 C_u^\frac{1}{2}\max_j K_j \|\psi_1\|_{L^6(\Omega)} \|\nabla \overline{\psi}\|_{L^2(\Omega)} ^{\frac12}  \|\mathcal{L} \overline{\psi}\|_{L^2(\Omega)}^{\frac32} \notag \\ 
& \quad \leq \frac{1}{8}\|\mathcal{L} \overline{\psi}\|_{L^2(\Omega)}^2+ 864\big[(1+C_p) C_2 C_u^\frac{1}{2}\max_j K_j\big]^4 \|\psi_1\|_{L^6(\Omega)}^4 \|\nabla \overline{\psi}\|_{L^2(\Omega)}^2, 
\label{dim-3}
\end{align}
and 
\begin{align}
& \left|\int_\Omega \overline{\Bu}\cdot\nabla \psi_2 \L\overline{\psi}\,\dd \Bx\right|  \notag \\
& \quad \leq   \|\overline{\Bu}\|_{L^6(\Omega)} \|\nabla \psi_2\|_{L^3(\Omega)}  \|\mathcal{L} \overline{\psi}\|_{L^2(\Omega)}  \notag \\
& \quad \leq   (1+C_p)\max_j K_j \|\overline{\psi}\|_{L^6(\Omega)} \|\nabla \psi_2\|_{L^3(\Omega)}  \|\mathcal{L} \overline{\psi}\|_{L^2(\Omega)}   \notag \\
& \quad \leq C_1 (1+C_p)\max_j K_j \|\nabla \overline{\psi}\|_{L^2(\Omega)} \|\nabla \psi_2\|_{L^3(\Omega)}  \|\mathcal{L} \overline{\psi}\|_{L^2(\Omega)}  \notag \\
& \quad \leq \frac{C_1 (1+C_p)\max_j K_j}{(\min_j D_j)^\frac12}   \|\nabla \psi_2\|_{L^3(\Omega)}  \|\mathcal{L} \overline{\psi}\|_{L^2(\Omega)}^\frac32\|\overline{\psi}\|^\frac12_{L^2(\Omega)}  \notag \\
&\quad \leq \frac{1}{8}\|\mathcal{L} \overline{\psi}\|_{L^2(\Omega)}^2+ 864 \frac{[C_1 (1+C_p)\max_j K_j]^4}{(\min_j D_j)^2} \|\nabla \psi_2\|_{L^3(\Omega)}^4 \| \overline{\psi}\|_{L^2(\Omega)}^2  \notag \\
& \quad \leq \frac{1}{8}\|\mathcal{L} \overline{\psi}\|_{L^2(\Omega)}^2+ 864 \frac{[C_1C_2C_u (1+C_p)\max_j K_j]^4}{(\min_j D_j)^2} \|\mathcal{L} \psi_2\|_{L^2(\Omega)}^2\|\nabla \psi_2 \|^2_{L^2(\Omega)}\| \overline{\psi}\|_{L^2(\Omega)}^2.
\label{dim-4}
\end{align}
Inserting \eqref{dim-2}--\eqref{dim-4} into \eqref{dim-1}, we obtain 
\begin{equation}\label{dim-5}
\frac{\dd}{\dd t}\|\sqrt{D}\nabla \overline{\psi}\|_{L^2(\Omega)}^2 + \|\mathcal{L}\overline{\psi}\|_{L^2(\Omega)}^2 
\leq M_{10} \|\sqrt{D}\nabla \overline{\psi}\|_{L^2(\Omega)}^2 +M_{10}\|\mathcal{L} \psi_2\|_{L^2(\Omega)}^2\|\overline{\psi}\|^2,
\end{equation}
where $M_{10}$ is a positive constant that depends only on $D_j,K_j,\Omega,c_\Delta,\delta$, $M_1$, $M_5$. Multiplying \eqref{dim-5} by $t>0$ and integrating over $[0,t]\subset [0,\infty)$, we get 
\begin{equation}\label{dim-6}
\begin{aligned}
   & t \|\sqrt{D}\nabla \overline{\psi}(t)\|_{L^2(\Omega)}^2 \\
   &\quad \leq \int_0^t (1+M_{10}\tau) \|\sqrt{D}\nabla \overline{\psi}(\tau)\|_{L^2(\Omega)}^2\,\mathrm{d}\tau 
   + M_{10} \int_0^t \tau \|\mathcal{L} \psi_2(\tau)\|_{L^2(\Omega)}^2 \|\overline{\psi}(\tau)\|_{L^2(\Omega)}^2\,\mathrm{d}\tau
   \\
   &\quad \leq (1+M_{10}t)\int_0^t \|\sqrt{D}\nabla \overline{\psi}(\tau)\|_{L^2(\Omega)}^2\,\mathrm{d}\tau
   + M_{10} tM_7(t) \sup_{\tau\in[0,t]}\|\overline{\psi}(\tau)\|_{L^2(\Omega)}^2,
\end{aligned}
\end{equation}
which combined with \eqref{S-conti-Hs} and \eqref{dim-0} yields the conclusion \eqref{smooth}.
\end{proof}
\begin{remark}
For the case $\varepsilon\in (0,\varepsilon_0]$, using the same argument as for Lemma \ref{lem-smooth}, we can obtain a similar estimate:
\begin{equation}\label{smooth-ep}
\|\nabla \psi_1^\varepsilon (t)- \nabla \psi_2^\varepsilon (t)  \|^2_{L^2(\Omega)}
\leq M_{11}\left(\frac{1}{t}+t^2\right)e^{M_9 t}\|\L^{\frac{s}{2}}(\psi_{1,0}^\varepsilon -\psi_{2,0}^\varepsilon)\|_{L^2(\Omega)}^2, \quad \forall\, t>0.
\end{equation}
A careful examination of the estimates shows that the coefficients in the estimate \eqref{smooth-ep} are independent of the parameter $\varepsilon$. 
\end{remark}

\begin{proof}[Proof of Theorem \ref{finite-dim}]
When $\varepsilon=0$, we take 
$$
\mathcal{S}=S_0(1),\quad E_1=\mathcal{V},\quad E= \mathcal{H}^s,\quad X=\mathscr{A}_0.
$$
It follows from Lemma \ref{lem-smooth} that 
\begin{equation}\notag 
\|\nabla \psi_1 (1)- \nabla \psi_2(1)  \|^2_{L^2(\Omega)}
\leq 2 M_{11}e^{M_9} \|\L^{\frac{s}{2}}(\psi_{1,0}-\psi_{2,0})\|_{L^2(\Omega)}^2.
\end{equation}
By the abstract result Theorem \ref{zelik-thm}, we can conclude that $\mathrm{dim}_\mathrm{f}(\mathscr{A}_0,\mathcal{H}^s)$ is finite as described in \eqref{A-dim} with the constant 
$$
C_A=\left(2 M_{11}e^{M_9}\right)^\frac12. 
$$ 
Thanks to \eqref{smooth-ep}, similar conclusion holds for $\mathscr{A}_\varepsilon$, $ \varepsilon\in (0,\varepsilon_0]$, with the same constant $C_A$. Hence, the family of global attractors $\{ \mathscr{A}_\varepsilon\}_{\varepsilon\in [0,\varepsilon_0]}$ have the same upper bound for their fractal dimensions. The proof is complete. 
\end{proof}
\begin{remark}
The upper bound for the fractal dimensions of $\{\mathscr{A}_\varepsilon\}_{\varepsilon\in [0,\varepsilon_0]}$ obtained in Theorem \ref{finite-dim} may not be sharp, compared to the volume contraction method (see \cite{Temam1997} and the references therein). However, the latter approach requires the semigroup $S_\varepsilon(t)$ to be uniformly quasi-differentiable with respect to the initial data, which is not available here due to the low regularity of the solutions. 
\end{remark}

\section{Conclusion}

The long-time validity of simplified models is a central issue, as they are frequently used for predictive purposes in complex physical systems. In this work, we establish the long-time validity of the three-dimensional sharp-interface Darcy--Boussinesq model for convection in layered porous media by demonstrating that it arises as the zero-thickness limit of the more physically realistic diffuse-material-interface formulation. This validation has been obtained in terms of the upper semi-continuity of global attractors, invariant measures, long-time averaged statistics such as the Nusselt number, and a uniform upper bound on the fractal dimension of the attractors. Our results extend the finite-time convergence analysis of \cite{DW2025}, where the behavior of solutions and the near-singular structure of the velocity field were examined.

The analysis relies on deriving uniform-in-time bounds, independent of the thickness of the transition layer, in a space compactly embedded within the natural phase space. These estimates are obtained through a Hopf--Constantin--Doering type framework \cite{doering1998jfm, Wang2008} adapted to the layered setting, combined with a careful choice of weighted function spaces suited to the models under consideration, as well as a novel interpolation space that is associated with a fractional power of the principal linear elliptic operator with discontinuous coefficients. While the present work focuses on straight interfaces, the extension to curved interfaces can be carried out analogously by employing curvilinear coordinates, provided that the interfaces are sufficiently smooth and well separated.

Although we have established upper semi-continuity of the long-time dynamics, several questions remain open. A natural direction is to investigate whether full continuity can be obtained, particularly in regimes where the dynamics is turbulent or intrinsically stochastic. A more detailed characterization of global attractors, especially near the critical Rayleigh–Darcy threshold for the onset of convection \cite{MW2013}, also deserves further study. From a statistical point of view, determining the rate of convergence of invariant measures would be of significant practical importance. Understanding the dependence of the Nusselt number on the Rayleigh number remains a central issue in convection theory \cite{doering1998jfm, DPZS2022, HFCJ2012}. Finally, since physical systems are inevitably subject to randomness, it is of considerable interest to analyze how small stochastic perturbations affect the long-time dynamics \cite{Young2002}.

\section*{Acknowledgments}
Wu is a member of the Key Laboratory of Mathematics for Nonlinear Sciences (Fudan University), Ministry of Education of China.


\appendix
\renewcommand{\thesection}{\Alph{section}}

\section{Some elementary lemmas}\label{sec-app}

In the appendix, we collect some elementary lemmas frequently used in previous proofs. First, we give a  Poincar\'e type inequality in the layer domain.
\begin{lemma}\label{lemmaA5}
  Let $E= (0,L)^2\times (0,h)$ be a layer domain in $\R^3$, $L, h>0$. Then it holds that
  \begin{equation}\label{A5-1}
    \|\psi\|_{L^2(E)} \leq h \|\partial_z\psi\|_{L^2(E)}
  \end{equation}
  for any $\psi\in H^1(\Omega)$ vanishing on the boundary $[0,L]^2\times \{0\}$ (or $[0,L]^2\times \{h\}$).
\end{lemma}
\begin{proof}
It suffices to prove \eqref{A5-1} for smooth functions. For any $(x,y,z)\in E$, using the fundamental theorem of calculus and H\"older inequality, we get 
\begin{equation} \nonumber
  |\psi(x,y,z)|^2 =  \left( \psi(x,y,0) + \int_0^z \partial_z \psi(x,y,s) \dd s \right)^2
  \leq  \left(\int_0^h |\partial_z \psi| \dd s\right)^2 \leq h\int_0^h |\partial_z \psi|^2 \dd s. 
\end{equation}
Integrating the above inequality over $E$, we see that 
\begin{equation}\nonumber
\|\psi\|_{L^2(E)}^2 \leq  h\int_0^h\dd z \int_{(0,L)^{2}}\dd x\dd y \int_0^h |\partial_z \psi|^2 \dd s 
=  h^2 \|\partial_z\psi\|_{L^2(E)}^2,
\end{equation}
which leads to the conclusion. 
\end{proof}
  

The following lemma provides some estimates involving the space $\mathcal{W}$ (see \cite[Lemma 2.2]{CNW2025}).
\begin{lemma}\label{W-embedding}
There exist two constants $C_l, C_u>0$ depending only on $H$ and $D_j$ such that
\begin{equation} \label{A4-0-1}
C_l\|\mathcal{L}\psi\|_{L^2(\Omega)} \leq \|\psi\|_{\mathcal{W}} \leq C_u\|\mathcal{L}\psi\|_{L^2(\Omega)},\quad \forall\, \psi \in \mathcal{W}.
\end{equation} 
%
\end{lemma} 

Analogously, we have 
\begin{lemma}\label{lemma-equivalence}
Let $0<\varepsilon\ll 1$ be given. There exist two constants $C_l, C_u>0$ depending only on $H$ and $D_j$ such that
\begin{equation}\label{A4-0}
C_l\| \L^\varepsilon \psi \|_{L^2(\Omega)} \leq \| \psi \|_{\mathcal{W}^\varepsilon} \leq C_u\| \L^\varepsilon \psi \|_{L^2(\Omega)}, \quad \forall\,\psi \in \mathcal{W}^\varepsilon,
\end{equation}  
\end{lemma}
\begin{proof} 
It is easy to check that 
\begin{equation}\nonumber
 \|\L^\varepsilon \psi \|_{L^2(\Omega)} 
 \leq  \| D^\varepsilon \partial_x^2 \psi \|_{L^2(\Omega)} 
 + \|  D^\varepsilon \partial_y^2 \psi \|_{L^2(\Omega)} 
 + \| \partial_z(D^\varepsilon \partial_z\psi) \|_{L^2(\Omega)} \leq (1+2\max_{j}D_j)\|\psi\|_{\mathcal{W}^\varepsilon}.
\end{equation}
for any $\psi \in \mathcal{W}^\varepsilon$. Next, for any $\psi\in \mathcal{W}^\varepsilon$, we denote $f := \L^\varepsilon \psi \in L^2(\Omega)$.
%
Then $\psi$ can be viewed as the unique strong solution to the elliptic problem
\begin{equation}\label{model1}
\left\{\begin{aligned}
  &\L^\varepsilon\psi = f &&\text{ in }\Omega,\\
  &\psi=0&& \text{for}\ z=0,-H, \quad \text{and periodic in the horizontal directions } (x,y).
\end{aligned}\right.
\end{equation}
Testing equation $\L^\varepsilon \psi =f$ by $\psi$ and using integration by parts, we get 
\begin{equation}\nonumber
\|\sqrt{D^\varepsilon} \nabla \psi \|_{L^2(\Omega)}^2\leq \| f \|_{L^2(\Omega)}\|\psi\|_{L^2(\Omega)}.  
\end{equation}
This combined with Lemma \ref{lemmaA5} yields 
\begin{equation}\label{A4-1}
 \| \nabla \psi \|_{L^2(\Omega)} \leq \frac{H}{\min_{j}D_j}\| f \|_{L^2(\Omega)},\quad \|\psi \|_{L^2(\Omega)}  \leq \frac{H^2}{\min_{j}D_j}\| f \|_{L^2(\Omega)}.
\end{equation} 
Next, testing equation $\L^\varepsilon \psi =f$ by $-\partial_x^2\psi$ and using integration by parts, we get
$$
(D^\varepsilon \nabla \partial_x\psi, \nabla \partial_x \psi)= -(f, \partial_x^2\psi), 
$$
which together with H\"{o}lder's and Young's inequalities yields that 
\begin{equation}\label{A4-2}
\| \nabla \partial_x \psi \|_{L^2(\Omega)} \leq  \frac{1}{\min_{j}D_j}\|f\|_{L^2(\Omega)} .
\end{equation}
Then we have 
\begin{equation}\label{A4-3}
\| \partial_x (D^\varepsilon \partial_z \psi) \|_{L^2(\Omega)} = \| D^\varepsilon \partial_z \partial_x \psi \|_{L^2(\Omega)} \leq \max_j D_j \| \partial_z\partial_x \psi \|_{L^2(\Omega)} \leq   \frac{\max_j D_j }{\min_{j}D_j}\|f\|_{L^2(\Omega)}. 
\end{equation}
A similar estimate holds for $\| \partial_y (D^\varepsilon \partial_z \psi) \|_{L^2(\Omega)}$. 
Finally, using the expression $-\partial_z (D^\varepsilon \partial_z \psi) = D^\varepsilon \partial_x ^2 \psi + D^\varepsilon \partial_y ^2 \psi + f$, we get 
\begin{equation}\label{A4-4}
\| \partial_z (D^\varepsilon \partial_z \psi) \|_{L^2(\Omega)} \leq  \left(\frac{2\max_j D_j}{\min_j D_j} +1\right)\| f \|_{L^2(\Omega)}.  
\end{equation}
Combining \eqref{A4-1}--\eqref{A4-4}, we arrive at the conclusion \eqref{A4-0}. The proof is complete.
\end{proof}
 
\begin{lemma}\label{lemmaA1}
Suppose $r\in [2,6]$. For any $\psi \in \mathcal{V}$, it holds that
   \begin{equation}\label{A1-0}
\|\psi\|_{L^r(\Omega)} \leq C_1 \|\psi\|_{L^2(\Omega)}^{ \frac{3}{r} -\frac12}\|\nabla\psi\|_{L^2(\Omega)}^{\frac32-\frac{3}{r}}.
  \end{equation}
Moreover, we have 
  \begin{align}\label{A1-1}
  & \|\nabla \psi\|_{L^r(\Omega)} \leq C_2\|\nabla\psi\|_{L^2(\Omega)}^{ \frac{3}{r} -\frac12}\|\psi\|_{\mathcal{W}}^{\frac32-\frac{3}{r}}, \quad \forall\,\psi\in \mathcal{W},
  \\
  \label{A1-2}
  &\|\nabla \psi\|_{L^r(\Omega)} \leq C_2\|\nabla\psi\|_{L^2(\Omega)}^{ \frac{3}{r} -\frac12}\|\psi\|_{\mathcal{W}^\varepsilon}^{\frac32-\frac{3}{r}}, \quad \forall\,\psi\in \mathcal{W}^\varepsilon.
  \end{align}
Here, the positive constants $C_1$, $C_2$ depend only on $r$, $\Omega$ and $D_j$.  
\end{lemma}
\begin{proof}
The cases $r=2,6$ are standard (see \cite[Lemma 2.3]{CNW2025}). For any $\psi \in \mathcal{V}$ and $r\in (2,6)$, it follows from the classical Gagliardo--Nirenberg inequality that 
\begin{equation}
  \|\psi\|_{L^r (\Omega)} \leq C\|\nabla\psi\|_{L^2(\Omega)}^{\frac32-\frac{3}{r}}\|\psi\|_{L^2(\Omega)}^{ \frac{3}{r} -\frac12} 
  + C\|\psi\|_{L^2(\Omega)}.
  \notag 
\end{equation}
Since $\psi$ vanishes at the top and bottom of $\partial \Omega$, we infer from Lemma \ref{lemmaA5} that 
\begin{equation*}
\|\psi\|_{L^r(\Omega)}\leq C(1+H^{\frac32-\frac3r})\|\nabla\psi\|_{L^2(\Omega)}^{\frac32-\frac{3}{r}}\|\psi\|_{L^2(\Omega)}^{ \frac{3}{r} -\frac12} =: C_1 \|\nabla\psi\|_{L^2(\Omega)}^{\frac32-\frac{3}{r}}\|\psi\|_{L^2(\Omega)}^{ \frac{3}{r} -\frac12},
\end{equation*}
which gives \eqref{A1-0}. Next, we consider $\psi\in \mathcal{W}$. Set $\widetilde{\nabla}\psi := (\partial_x \psi,\partial_y \psi, D\partial_z\psi)$. By definition, it holds  $\widetilde{\nabla}\psi \in H^1(\Omega)$ and  
\begin{equation}
\|\widetilde{\nabla}\psi\|_{H^1(\Omega)}  \leq \|\psi\|_{\mathcal{W}}.
\notag 
\end{equation}
Applying the Gagliardo--Nirenberg inequality to $\widetilde{\nabla}\psi$, we get 
\begin{equation*}
\begin{aligned}
\|\widetilde{\nabla}\psi\|_{L^r(\Omega)}
&\leq  C\|\widetilde{\nabla}\psi\|_{L^2(\Omega)}^{ \frac{3}{r} -\frac12}\|\nabla(\widetilde{\nabla}\psi)\|_{L^2(\Omega)}^{\frac32-\frac{3}{r}}
+ C\|\widetilde{\nabla}\psi\|_{L^2(\Omega)}\\
&\leq C\|\widetilde{\nabla}\psi\|_{L^2(\Omega)}^{ \frac{3}{r} -\frac12}\|\widetilde{\nabla}\psi\|_{H^1(\Omega)}^{\frac32-\frac{3}{r}} + C\|\widetilde{\nabla}\psi\|_{L^2(\Omega)}\\
& \leq C\|\widetilde{\nabla}\psi\|_{L^2(\Omega)}^{ \frac{3}{r} -\frac12}\|\widetilde{\nabla}\psi\|_{H^1(\Omega)}^{\frac32-\frac{3}{r}}\\
& \leq C\|\widetilde{\nabla}\psi\|_{L^2(\Omega)}^{ \frac{3}{r} -\frac12}\|\psi\|_{\mathcal{W}}^{\frac32-\frac{3}{r}},
\end{aligned}
\end{equation*}
where $C>0$ depends on $r$ and $\Omega$. Using the fact  
\begin{equation*}
 \min_{j}D_j \|\nabla\psi\|_{L^r(\Omega)} \leq \|\widetilde{\nabla}\psi\|_{L^r(\Omega)} \leq \max_{j}D_j \|\nabla\psi\|_{L^r(\Omega)},
\end{equation*}
we can conclude \eqref{A1-1} with $C_2 = \frac{C\max_{j}D_j^{ \frac{3}{r} -\frac12}}{\min_{j} D_j}$.

Similarly, for $\psi\in \mathcal{W}^\varepsilon$, we define $ \nabla^\varepsilon\psi := (\partial_x \psi,\partial_y \psi, D^\varepsilon\partial_z\psi)$ so that  
\[\|\nabla^\varepsilon\psi\|_{H^1(\Omega)} 
\leq \|\psi\|_{\mathcal{W}^\varepsilon}.\]  
Applying the Gagliardo--Nirenberg inequality to $\nabla^\varepsilon\psi$ gives
\begin{equation*} 
\begin{aligned}
\|\nabla^\varepsilon\psi\|_{L^r(\Omega)} \leq C\|\nabla^\varepsilon\psi\|_{L^2(\Omega)}^{ \frac{3}{r} -\frac12}\|\psi\|_{\mathcal{W}^\varepsilon}^{\frac32-\frac{3}{r}}. 
\end{aligned}
\end{equation*}
Noticing that 
\begin{equation}\nonumber
   \min_{j}D_j \|\nabla\psi\|_{L^r(\Omega)} \leq \|\nabla^\varepsilon\psi\|_{L^r(\Omega)} \leq \max_{j}D_j \|\nabla\psi\|_{L^r(\Omega)} ,
\end{equation}
we obtain \eqref{A1-2}. The proof is complete. 
\end{proof}

\section{Fractional Sobolev spaces}\label{sec-appB}
Based on the expansion of eigenfunctions of the operator $\mathcal{L}$, for any $\psi = \sum_{k=1}^\infty a_k w_k$, we define 
\[ \mathcal{L}^s\psi := \sum_{k=1}^\infty \lambda_k^s a_k w_k,\quad s\in \mathbb{R}.\]
Then we define the space $\mathcal{H}^s=\mathrm{dom}(\mathcal{L}^\frac{s}{2})$ with the norm given by 
\begin{equation}\nonumber
\|\psi\|_{\mathcal{H}^s}^2 :=\sum_{k=1}^\infty \lambda_k^s|a_k|^2. 
\end{equation}

\begin{lemma}\label{lem-B1}
For any $-\infty<s_1<s_2 <\infty$ and any $\psi$ satisfying $\|\psi\|_{\mathcal{H}^{s_2}}<\infty$, we have 
\begin{equation}\nonumber
\|\psi\|_{\mathcal{H}^{s_1}} \leq \lambda_1^{s_1-s_2}\|\psi\|_{\mathcal{H}^{s_2}},
\end{equation}
and 
\begin{equation}\nonumber
  \|\psi\|_{\mathcal{H}^{s}} \leq \|\psi\|_{\mathcal{H}^{s_1}}^\theta\|\psi\|_{\mathcal{H}^{s_2}}^{1-\theta},~~~~\text{ for any } s= \theta s_1 + (1-\theta)s_2\text{ with }\theta\in (0,1).
\end{equation} 
\end{lemma}
\begin{proof}
According to \cite{CNW2025}, the eigenvalues $\{\lambda_k\}$ of the self-adjoint positive operator $\mathcal{L}$ satisfies
\[0 <\lambda_1\leq \lambda_2\leq \cdots.\]
Thus,
\begin{equation}\nonumber
  \|\psi\|_{\mathcal{H}^{s_1}} ^2 = \sum_{k=1}^\infty \lambda_k^{s_1} |a_k|^2  = \sum_{k=1}^\infty \frac{\lambda_k^{s_2} }{\lambda_k^{s_2-s_1}} |a_k|^2\leq \frac{1}{\lambda_1^{s_2-s_1}}\sum_{k=1}^\infty \lambda_k^{s_2} |a_k|^2=  \lambda_1^{s_1-s_2}\|\psi\|_{\mathcal{H}^{s_2}} ^2.
\end{equation}
The interpolation inequality follows from the discrete version of H\"older's inequality,
\begin{equation}\nonumber
\begin{aligned}
  \|\psi\|_{\mathcal{H}^{s}}^2 
  & =   \sum_{k=1}^\infty \lambda_k^{\theta s_1+(1-\theta)s_2} |a_k|^{2\theta + 2(1-\theta)   } 
  \leq  \left(  \sum_{k=1}^\infty \lambda_k^{s_1} |a_k|^{2}\right) ^\theta \left(  \sum_{k=1}^\infty \lambda_k^{s_2} |a_k|^{2 }\right)^{1-\theta} \\
  & = \|\psi\|_{\mathcal{H}^{s_1}}^{2\theta}\|\psi\|_{\mathcal{H}^{s_2}}^{2(1-\theta)}.
  \end{aligned}
\end{equation}
The proof is complete.
\end{proof}

\begin{lemma}\label{lemmaB2}
The norms $\|\cdot\|_{\mathcal{H}^0}$,  $\|\cdot\|_{\mathcal{H}^1}$ and $\|\cdot\|_{\mathcal{H}^2}$ are equivalent to the norms $\|\cdot\|_{L^2(\Omega)}$, $\|\cdot\|_{\mathcal{V}}$ and $\|\cdot\|_{\mathcal{W}}$, respectively.
\end{lemma}
\begin{proof}
Any $\psi\in \mathcal{H}$ can be written as a series in terms of the orthonormal basis $\{w_k\}$ in $\mathcal{H}$ such that  
\[ \psi = \sum_{k=1}^\infty (\psi, w_k) w_k.\]
Then its $L^2$-norm is given by Parseval's identity
\begin{equation}
  \|\psi\|_{L^2(\Omega)}^2 = \sum_{k=1}^\infty |(\psi,w_k)|^2.
  \notag 
\end{equation}
This implies that $\|\cdot\|_{L^2(\Omega)}$ and $\|\cdot\|_{\mathcal{H}^0}$ are equivalent norms on $\mathcal{H}$.

For the case $s= 1$, since $\{w_k\}$ is also a orthogonal basis in $\mathcal{V}$, any $\psi \in \mathcal{V}$ can be written as a series 
\[ \psi =  \sum_{k=1}^\infty (\psi, w_k)_1 w_k := \sum_{k=1}^\infty (D\nabla \psi, \nabla w_k) w_k ,\]
which converges in $\mathcal{V}$. We have the norm 
\begin{equation}
\|\psi\|_{\mathcal{V}}^2 = \sum_{k=1}^\infty |(\psi, w_k)_1|^2  (w_k, w_k)_1  = \sum_{k=1}^\infty |(\psi, w_k)_1|^2 (\mathcal{L}w_k, w_k) = \sum_{k=1}^\infty |(\psi, w_k)_1|^2 \lambda_k, \notag 
\end{equation}
which is also equivalent to the norm $\|\psi \|_{\mathcal{H}^1}$.

Finally, any $\psi\in\mathcal{W} \subset \mathcal{H}$ can be expanded as $\psi = \sum_{k=1}^\infty a_k w_k$, where the series converges at least in $L^2(\Omega)$.  Then we have 
\begin{equation}
\mathcal{L}\psi =   \sum_{k=1}^\infty a_k \mathcal{L}w_k =  \sum_{k=1}^\infty a_k \lambda_k w_k. 
\notag 
\end{equation}
Consequently,
\begin{equation}
 \|\mathcal{L}\psi\|_{L^2(\Omega)}^2= \sum_{k=1}^\infty  \lambda_k^2|a_k|^2= \|\psi\|_{\mathcal{H}^2}^2. \notag 
\end{equation} 
These, together with Lemma \ref{lemma-equivalence}, show that $\|\cdot\|_{\mathcal{H}^2}$ and $\|\cdot\|_{\mathcal{W}}$ are equivalent norms on $\mathcal{W}$.
\end{proof}

We have the following equivalence on the fractional norm $\mathcal{H}^s$ and the classical Sobolev norm $H^s(\Omega)$, $s\in (0,1)$, whose proof is inspired by \cite[Appendix B]{BZ2000}.

\begin{lemma}\label{lem-B2}
Let $s\in (0,1)$. The fractional $\mathcal{H}^s$-norm can be interpreted as the interpolation norm between $\mathcal{H}$ and $\mathcal{V}$, and is equivalent to the fractional Sobolev norm $\|\cdot\|_{H^s(\Omega)}$. Furthermore, the fractional $\mathcal{H}^{1+s}$-norm can be interpreted as the interpolation norm between $\mathcal{V}$ and $\mathcal{W}$.
\end{lemma}
\begin{proof}
Let us introduce the interpolation norm between $\mathcal{H}$ and $\mathcal{V}$ using the K-method \cite{Triebel}:
\begin{equation}\nonumber
\|\psi\|_{s,\mathcal{H},\mathcal{V}} := \left(\int_0^\infty t^{-2s} K^2(t,\psi,\mathcal{H},\mathcal{V}) \frac{\dd t}{t}\right)^\frac12,~~~~\text{ for }s\in (0,1),
\end{equation}
where  the $K$-functor is given by 
\[ K^2(t,\psi,\mathcal{H},\mathcal{V}) :  = \inf_{\psi_1 \in \mathcal{V}}  \|\psi-\psi_1\|_{L^2(\Omega)}^2+ t^2\|\psi_1\|_{\mathcal{V}}^2.\]
Let 
\[\psi = \sum_{k=1}^\infty a_k w_k ~~~~\text{ and }~~~~\psi_1 = \sum_{k=1}^\infty b_k w_k.\]
Then 
\[\|\psi-\psi_1\|_{L^2(\Omega)}^2+ t^2\|\psi_1\|_{\mathcal{V}}^2  = \sum_{k=1}^\infty (a_k - b_k)^2+t^2 b_k^2 \|w_k\|_{\mathcal{V}}^2 = \sum_{k=1}^\infty (a_k - b_k)^2+t^2 b_k^2  \lambda_k. \] 
Solving the minimization problem gives 
\begin{equation}
K^2(t,\psi,\mathcal{H},\mathcal{V})  = \sum_{k=1}^\infty\left. (a_k - b_k)^2+t^2 b_k^2  \lambda_k\right|_{b_k = \frac{a_k^2}{t^2 \lambda_k +1}} = \sum_{k=1}^\infty \frac{a_k^2t^2\lambda_k }{1+t^2 \lambda_k}. \notag 
\end{equation} 
Consequently, 
\begin{equation}\nonumber
 \int_0^\infty t^{-2s}  K^2(t,\psi,\mathcal{H},\mathcal{V})  \frac{\dd t}{t}  = \sum_{k=1}^\infty   \int_0^\infty t^{1-2s}   \frac{a_k^2 \lambda_k }{1+t^2 \lambda_k} \dd t  =  \left(\frac{2}{\pi}\sin \pi s\right)^{-1}\sum_{k=1}^\infty \lambda_k^s |a_k|^2.
\end{equation}
Hence, $\|\psi\|_{\mathcal{H}^s} =\sqrt{\frac{2}{\pi}\sin \pi s}\|\psi\|_{s,\mathcal{H},\mathcal{V}}$ for any $\psi \in \mathcal{V}$.

Similarly, the interpolation norm between $\mathcal{V}$ and $\mathcal{W}$ is defined as 
\begin{equation}\nonumber
\|\psi\|_{s,\mathcal{V},\mathcal{W}} := \left(\int_0^\infty t^{-2s} K^2(t,\psi,\mathcal{V},\mathcal{W}) \frac{\dd t}{t}\right)^\frac12,~~~~\text{ for }s\in (0,1),
\end{equation}
where the $K$-functor becomes
\[ K^2(t,\psi,\mathcal{V},\mathcal{W}):  = \inf_{\psi_1 \in \mathcal{W}}  \|\psi-\psi_1\|_{\mathcal{V}}^2+ t^2\|\psi_1\|_{\mathcal{W}}^2.\]
By Lemma \ref{lemma-equivalence}, we can use the following equivalent form of $K^2(t,\psi,\mathcal{V},\mathcal{W})$: 
\[ \inf_{\psi_1 \in \mathcal{W}}  \|\psi-\psi_1\|_{\mathcal{V}}^2+ t^2\|\mathcal{L}\psi_1\|_{L^2(\Omega)}^2.\]
In a similar way, we have
\[ \inf_{\psi_1 \in \mathcal{W}}  \|\psi-\psi_1\|_{\mathcal{V}}^2+ t^2\|\mathcal{L}\psi_1\|_{L^2(\Omega)}^2 = \sum_{k=1}^\infty\left. (a_k - b_k)^2\lambda_k +t^2 b_k^2 \lambda_k^2\right|_{b_k = \frac{a_k^2}{t^2 \lambda_k +1}} = \sum_{k=1}^\infty \frac{a_k^2t^2\lambda_k^2 }{1+t^2 \lambda_k}, \] 
and hence,
\[
\sum_{k=1}^\infty   \int_0^\infty t^{1-2s}   \frac{a_k^2 \lambda_k^2  }{1+t^2 \lambda_k} \dd t  =  \left(\frac{2}{\pi}\sin \pi s\right)^{-1}\sum_{k=1}^\infty \lambda_k^{s+ 1} |a_k|^2 = \left(\frac{2}{\pi}\sin \pi s\right)^{-1} \|\psi\|_{\mathcal{H}^{1+s}}^2.\]
This implies the equivalence of  $\|\psi\|_{s, \mathcal{V},\mathcal{W}}$ and $\|\psi\|_{\mathcal{H}^{1+s}} $ for any $\psi \in \mathcal{W}$. 
\end{proof}

We note that $\|\cdot\|_{\mathcal{H}^{1+s}}$ is not equivalent to $\|\cdot\|_{H^{1+s}(\Omega)}$, since the space $\mathcal{W}$ is different from the classical $H^2(\Omega)$ space in general, see \cite{CNW2025}. Nevertheless, we have the following result

\begin{lemma}\label{lem-B3}
For any $s\in (0,1)$ and $j=1,2,\cdots,l$, there exists a constant $C>0$ such that it holds that
\begin{equation}
 \|\psi\|_{H^{1+s}(\Omega_j)} \leq C\|\psi\|_{\mathcal{H}^{1+s}},\quad \forall\, \psi \in \mathcal{W}, \notag 
\end{equation}
where $\|\cdot\|_{H^{1+s}(\Omega_j)}$ is the fractional Sobolev norm on the layer $\Omega_j$.
\end{lemma}
\begin{proof}
For any $j=1,\cdots,l$, denote 
\[ \mathcal{V}|_{\Omega_j}=\{\psi|_{\Omega_j}:~\psi\in \mathcal{V}\}~~~~\text{ and }~~~~~\mathcal{W}|_{\Omega_j}=\{\psi|_{\Omega_j}:~\psi\in \mathcal{W}\}.\]
Then the interpolation norm between $\mathcal{V}|_{\Omega_j}$ and $\mathcal{W}|_{\Omega_j}$ is defined as
\[ \|\psi\|_{s, \mathcal{V}|_{\Omega_j}, \mathcal{W}|_{\Omega_j}} 
:=  \left(\int_0^\infty t^{-2s} K^2(t,\psi, \mathcal{V}|_{\Omega_j}, \mathcal{W}|_{\Omega_j}) \frac{\dd t}{t}\right)^\frac12,~~~~\text{ for }s\in (0,1),\]
where 
\begin{equation*}
K^2(t,\psi, \mathcal{V}|_{\Omega_j}, \mathcal{W}|_{\Omega_j}) 
= \inf_{\psi_1 \in \mathcal{W}|_{\Omega_j}}  \|\psi-\psi_1\|_{\mathcal{V}|_{\Omega_j}}^2 
+ t^2\|\psi_1\|_{\mathcal{W}|_{\Omega_j}}^2.
\end{equation*}
Noting that 
\begin{equation*} 
K^2(t,\psi, \mathcal{V}|_{\Omega_j}, \mathcal{W}|_{\Omega_j})
\leq\inf_{\psi_1\in \mathcal{W}}  \|\psi-\psi_1\|_{\mathcal{V}}^2 
+ t^2\|\psi_1\|_{\mathcal{W}}^2
=K^2(t,\psi, \mathcal{V},\mathcal{W}),
\end{equation*}
one has 
\begin{equation*}
  \|\psi\|_{s, \mathcal{V}|_{\Omega_j}, \mathcal{W}|_{\Omega_j}} \leq \|\psi\|_{s, \mathcal{V},\mathcal{W}},\quad \forall\,\psi\in \mathcal{W}.
\end{equation*}
Recall that any function $\psi$ in $\mathcal{W}$ is piecewise $H^2(\Omega)$. Then the restriction norm $\|\cdot\|_{\mathcal{W}|_{\Omega_j}}$ coincides with the classical Sobolev norm $\|\cdot\|_{H^2(\Omega_j)}$, and the interpolation norm $\|\cdot\|_{s, \mathcal{V}|_{\Omega_j}, \mathcal{W}|_{\Omega_j}}$ is also equivalent to the fractional Sobolev norm $\|\cdot\|_{\mathcal{H}^{1+s}(\Omega_j)}$. 
Hence,
\begin{equation}\nonumber
  \|\psi\|_{H^{1+s}(\Omega_j)} \leq C  \|\psi\|_{s, \mathcal{V}|_{\Omega_j}, \mathcal{W}|_{\Omega_j}} \leq C \|\psi\|_{s, \mathcal{V},\mathcal{W}} \leq C\|\psi\|_{\mathcal{H}^{1+s}},\quad \forall\, \psi \in \mathcal{W}.
\end{equation}
The proof is complete.
\end{proof}

\begin{remark}
Similar results hold for the operators $\mathcal{L}^\varepsilon$ and the spaces $\mathcal{H}_\varepsilon^s$, $\mathcal{W}^\varepsilon$. By construction of $D^\varepsilon$, the generic positive constant $C$ in the corresponding estimates is independent of the thickness of the transition layer $\varepsilon$.
\end{remark}

\end{document}